\documentclass[a4paper,reqno]{amsart}
\usepackage{amsfonts}
\usepackage{color}
\usepackage{hyperref}
\definecolor{ddorange}{rgb}{1,0.5,0}
\definecolor{ddcyan}{rgb}{0,0.2,1.0}

\newcommand{\BBB}{\color{black}}
\newcommand{\CCC}{\color{black}}
\newcommand{\EEE}{\color{black}}

\usepackage[textwidth=408pt,textheight=658pt,heightrounded, headsep=12pt,vmarginratio=1:1]{geometry}
\usepackage[T1]{fontenc} 
\usepackage[utf8]{inputenc}
\usepackage[french, italian, english]{babel}
\usepackage{amsmath,amssymb,amsthm, mathtools}
\usepackage{braket}
\usepackage{ae}
\usepackage{verbatim}
\usepackage{alltt}
\usepackage[titletoc]{appendix}

\newcommand{\C}{{\mathbb C}}
\newcommand{\N}{{\mathbb N}}

\newcommand{\R}{{\mathbb R}}

\newcommand{\EE}{\mathrm{E}}
\newcommand{\En}{E}
\newcommand{\A}{\mathrm{A}}
\newcommand{\tQ}{{\widetilde{Q}}}
\renewcommand{\Cap}{\mathrm{Cap}}
\newcommand{\Rn}{{\R}^n}
\newcommand{\Mb}{{M_b}}
\newcommand{\Mnn}{{\mathbb{M}^{n\times n}_{sym}}}
\newcommand{\MD}{{\mathbb M}^{n{\times}n}_D}
\newcommand{\dx}{{\, \mathrm{d}x}}
\newcommand{\dH}{{\, \mathrm{d}{\mathcal H}^{n-1}}}
\newcommand{\tr}{{\rm tr}}
\newcommand{\sm}{\setminus}
\newcommand{\hn}{{\mathcal H}^{n-1}}
\newcommand{\Ln}{{\mathcal{L}}^n}
\newcommand{\QQ}{{\mathcal Q}}
\newcommand{\LL}{{\mathcal L}}
\newcommand{\PP}{{\mathcal P}}
\newcommand{\HH}{{\mathcal H}}

\newcommand{\WW}{{H^1(\Omega)}}
\newcommand{\da}{{\| \nabla \alpha \|_{2}^{2}}}
\newcommand{\wtos}{\mathrel{\mathop{\rightharpoonup}\limits^*}}
\newcommand{\dom}{\partial \Omega}
\newcommand{\dod}{\partial_D \Omega}
\newcommand{\don}{\partial_N \Omega}
\newcommand{\diver}{\mathrm{div}\,}
\newcommand{\Sn}{\mathbb{S}^{n-1}}
\newcommand{\xy}{^\xi_y}
\newcommand{\MbDd}{{M_b(\Omega \cup \dod; \Mnn)}}

\newcommand{\Mbu}{{M_b(B_1; \Mnn)}}
\newcommand{\Lnn}{{L^2(\Omega; \Mnn)}}
\newcommand{\Lnnu}{{L^2(B_1; \Mnn)}}

\newcommand{\wto}{\rightharpoonup}
\newcommand{\wt}{\widetilde}

\newcommand{\V}{{\mathcal{V} }}
\newcommand{\WWi}{{H^1(\Omega; [0,1])}}
\newcommand{\tki}{t_k^i}
\newcommand{\tkim}{t_k^{i-1}}
\newcommand{\wki}{w_k^i}
\newcommand{\uki}{u_k^i}
\newcommand{\pki}{p_k^i}
\newcommand{\uk}{u_k}

\newcommand{\aki}{\alpha _k^i}

\newcommand{\ot}{{[0, T]}}

\newcommand{\ol}{\overline}
\newcommand{\E}{{\mathcal{E}}}

\newcommand{\mres}{\mathbin{\vrule height 1.6ex depth 0pt width
0.13ex\vrule height 0.13ex depth 0pt width 1.3ex}}

\theoremstyle{plain}
\begingroup
\newtheorem{thm}{Theorem}[section]
\newtheorem{lemma}[thm]{Lemma}
\newtheorem{prop}[thm]{Proposition}

\endgroup
\theoremstyle{definition}
\begingroup

\newtheorem{rem}[thm]{Remark}

\endgroup
\theoremstyle{remark}
\begingroup
\endgroup

\makeatletter
\newenvironment{taggedsubequations}[1]
 {%
  \addtocounter{equation}{-1}%
  \begin{subequations}%
  \def\@currentlabel{#1}%
  \renewcommand{\theequation}{#1.\arabic{equation}}%
 }
 {\end{subequations}}
 \makeatother

\numberwithin{equation}{section}

\begin{document}

\nocite{*}
\title[Energetic solutions for a coupled plasticity--damage model for geomaterials]{Energetic solutions for the coupling of associative
plasticity with damage in geomaterials}
\author{Vito Crismale}
\date{}
\begin{abstract}
\small{We prove existence of globally stable quasistatic evolutions, referred to as energetic solutions, for a model proposed by Marigo and Kazymyrenko in 2019. The behaviour of geomaterials under compression is studied through the coupling of Drucker-Prager plasticity model with a damage term tuning kinematical hardening. This provides a new approach to the modelling of geomaterials, for which non associative plasticity is usually employed. 
The kinematical hardening is null where the damage is complete, so there the behaviour is perfectly plastic. We analyse the model combining tools from the theory of capacity and from the treatment of linearly elastic materials with cracks.
}
\end{abstract}
\maketitle
{\small
\par
}
\setcounter{tocdepth}{1} 
 \tableofcontents
 
 \section{Introduction}
 
The theory of elasto-plasticity provides a common framework to model 
 the behaviour of 
materials displaying permanent deformations when a critical stress threshold is reached, even if these deformations could be caused by
many different physical mechanisms:
for instance, for crystalline materials, such as metals, plastic deformation is mainly due to dislocations, atomic defects inside the lattice, while for granular materials
like concrete, rocks, and soils it is mostly the result of relative sliding of the 
microparticles composing them.
Such intrinsic difference between materials
 reflects in the choice
 of the elasto-plastic model to adopt;
in particular, granular materials (we refer to them also as geomaterials) usually undergo permanent volumetric changes, in contrast to crystalline materials. 

The general setting of elasto-plasticity, in small-strain assumptions,
is based on an additive decomposition of the \emph{total strain} $\EE u \in \Mnn$, that is the symmetrized gradient of the \emph{displacement field} $u\colon \Omega\to \Rn$, $\Omega\subset\Rn$ 
 being the \emph{reference domain}, into \emph{elastic strain} and \emph{plastic strain}:
\[
\EE u=e+p.
\]
The \emph{stress tensor} $\sigma\colon \Omega\to \Mnn$ depends linearly on $e$ according to Hooke's law 
\[
\sigma=\C e,
\]
$\C$ being the fourth order positive definite Hooke's tensor, it is constrained to lie in a fixed closed and convex set $K \subset \Mnn$, referred to as \emph{constraint set}, and satisfies equilibrium conditions involving the external loads.
 When $\sigma$ is in the interior of $K$, the deformations are reversible (after a loading cycle the material has the same configuration), so the behaviour is elastic and no plasticity is produced. Conversely, a plastic flow may be produced when $\sigma \in \partial K$ (that is, $\sigma$ reaches the \emph{yield surface}): if plasticity flows in the cone of normal directions to $K$ at $\sigma\in \partial K$, that is $\dot{p}\in N_K(\sigma)$, the plasticity is said \emph{associative}, otherwise \emph{non associative} (e.g., if $\dot{p}\in N_{\widetilde{K}(\sigma)}(\sigma)$, for $\sigma \in \partial K$ and $K,\,\widetilde{K}(\sigma)\subset \Mnn$ different closed convex sets).

Associative plasticity is largely employed for crystalline materials. Moreover, these materials usually do not develop permanent volumetric changes: then 
$\frac{\mathrm{tr}\,p}{n}$, the hydrostatic part of $p$ related to irreversible volumetric deformation, is null and $K$ does not depend on the \emph{mean stress} $\sigma_m:=\frac{\mathrm{tr}\sigma}{n} \in \R$, but only on $\sigma_D:=\sigma-\sigma_m\, \mathrm{Id} \in \MD$  (here $\mathrm{tr}$ denotes the trace of a matrix and $\MD$ the deviatoric matrices, that is the matrices with null trace). For these materials thus $K\subset \MD$, and the corresponding associative models are said
of \emph{Prandtl-Reuss type} (we refer to \cite{Suq81, Mora} for their reference mathematical treatment).

Differently, granular materials display volumetric plastic deformations, depending on the hydrostatic pressure. In fact, an irreversible rearrangement of the microparticles may be caused by  applying a triaxial test, namely a shear compression plus a compression normal to the shear plane; as a result, for geomaterials the volume could increase even in compression, providing an example of \emph{dilatance}. In the shear compression, the volume may increase with that of the internal voids between the microparticles, interpreted as microcracks, due to a less efficient particle organisation; in the normal compression a part of these microcracks may be closed, preventing a free sliding and thus the free relaxation to the initial configuration in unloading.

Most of the models for geomaterials are formulated in the realm of non associative plasticity, see e.g.\ \cite{KM1, KM2, KM3, UllWamAleSamDegFra21}. (We refer the reader to the variational treatment of non associative plasticity in \cite{DMDSSol10, BabFraMor12, FraMor17} and to the related \cite{UllAleWamDegFra20}). The reason to put aside associative plasticity is that, when modifying Prandtl-Reuss plasticity accounting for the dependence of $K$ on $\sigma_m$, too important dilatancy effects appear, regardless of the form of $K$.

A different approach to the modelling of geomaterials
 has been 
 proposed by Marigo and Kazymyrenko in \cite{KazMar18}. Basing on a micromechanical analysis, they consider the coupling of an associative elasto-plastic model with damage, tuning kinematical hardening. 
The \emph{damage variable} $\alpha\colon [0,T]\to [0,1]$ is associated to the distribution density
of closed microcracks, and reflects into an internal blocked energy depending on plastic strain, 
being related
 to irreversible deformations.  
Assuming Coulomb law for the sliding with friction between the lips of closed microcracks, plasticity follows a Drucker–Prager law, that is
\[
K=\{\sigma\in \Mnn \colon \tau \sigma_m + |\sigma_D| -k\leq 0\},
\] 
for $\tau$, $k>0$. Then the plastic dissipation from a plastic strain $q$ is the relaxation of 
\[
\int_\Omega H(p(t)-q) \dx,\qquad H(\xi)=\sup_{\sigma\in K} \sigma \colon \xi
\]
for $p(t)$ 
a 
bounded Radon measure and the
 stored energy at time $t$ is the sum of the elastic energy of the sound material plus the kinematical hardening term (depending on $\alpha$)
\[
\QQ(e(t))+\tQ(\alpha(t),p(t))= \frac{1}{2}\int_{\Omega} \C e(x) : e(x) \dx + \int_\Omega \C_1(\alpha(x)) p(x) \colon p(x) \dx,
\]
for $\C_1$ a fourth-order tensor
 positive definite except for $\alpha=0$, corresponding to the fully damaged material, while formally  $\C_1$ is 
 $+\infty$ for $\alpha=1$, meaning that the sound material does not display plasticity.
The growth of the microcracks is modelled by the terms usually present in gradient damage models, corresponding to Ambrosio-Tortorelli approximation for a fixed thickness parameter
\[
D(\alpha(t))+\|\nabla \alpha(t)\|_{L^2}^2,
\]
cf.\ \cite{CRASII}, \cite[Section~4]{AMV}, \cite[Remark~3.1]{KazMar18}, and the coupling between Coulomb law for the sliding of crack lips with Griffith law for crack propagation in \cite{AndBamMar86, ZhuShaKon11}. Damage is assumed to be an irreversible process, consequently $\alpha$ is nonincreasing in time. The reference experiment is the triaxial test; here the external loading is
represented by a suitable duality between the displacement and a loading term $\LL(t)$. 
Remarkably, the simulations discussed in \cite{KazMar18} reproduce well the phenomena expected for geomaterials, in particular the dilatance.

The main result of the present work is the proof of existence of quasistatic evolutions for the model of \cite{KazMar18}, fulfilling the notion of
 \emph{energetic solutions} \`a la Mielke-Theil (cf.\ \cite{MieThe99MMRI, MieThe04RIHM, MieRou}). This notion is based on a global stability condition (qs1), which prescribes that at each time the current configuration minimises the sum of the total internal energy and the dissipation potential, and on an energy-dissipation balance (qs2) between the total variation in time of the internal energy, the total dissipated energy, and the work of the external loadings.
The precise assumptions and the functional framework of the existence result stated below, are described in Section~\ref{Sec:2}.  
\begin{thm}\label{thm:2305211857}
Assume the conditions \eqref{do}, \eqref{eq:D}, \eqref{eq:Celast}, \eqref{eq:C}, \eqref{0302211716}, \eqref{eq:data}, \eqref{2305211859}, and the definitions in Section~\ref{Sec:2}. Then there exists an evolution $\alpha\colon \ot \to H^1(\Omega)$, $u\colon \ot \to BD(\Omega)$, $e\colon \ot \to \Lnn$, $p\colon \ot \to \MbDd$, such that 
 \begin{itemize}
 \item[(qs0)] $\alpha(s)\leq \alpha(t)$ a.e.\ in $\Omega$ for every $0\leq s \leq t\leq T$;
\item[(qs1)] $(u(t), e(t), p(t))\in \A(w(t))$  for every $t\in \ot$ and
 \begin{equation*}
\begin{split}
\QQ(e(t))& +D(\alpha(t))+\| \nabla \alpha(t)\|_{L^2}^2 + \tQ(\alpha(t),p(t))  - \langle \LL(t), u(t) \rangle \\& \leq  \QQ(\eta)+D(\beta)+\| \nabla \beta\|_{L^2}^2 + \tQ(\beta, q)  +\HH(q-p(t)) - \langle \LL(t), v \rangle
\end{split}
\end{equation*}
for every $\beta \leq \alpha(t)$, $\beta\geq 0$, $(v,\eta,q)\in \A(w(t))$;
\item[(qs2)] $p\colon \ot \to \MbDd$ has bounded variation and for every $t\in\ot$
\begin{equation*}
\begin{split}
\QQ(e(t))& +D(\alpha(t))+\| \nabla \alpha(t)\|_{L^2}^2 + \tQ(\alpha(t),p(t)) + \V_\HH(p;0,t) - \langle \LL(t), u(t) \rangle \\
& = \QQ(e_0) +D(\alpha_0)+\| \nabla \alpha_0\|_{L^2}^2 + \tQ(\alpha_0,p_0) - \langle \LL_0, u_0 \rangle +\int_0^{t} \langle \sigma(s), \EE \dot{w}(s) \rangle \,\mathrm{d}s \\& \hspace{1em} - \int_0^{t} \{ \langle \dot{\LL}(s), u(s) \rangle + \langle \LL(s), \dot{w}(s) \rangle \} \,\mathrm{d}s.
\end{split}
\end{equation*} 
\end{itemize}
Moreover,  for every $t\in\ot$ except at most countable many, it holds that: 
$\alpha$ and $e$ are strongly continuous from $[0,T]$ into $H^1(\Omega)$ and $\Lnn$; $u$ and $p$ are weakly$^*$ continuous in $BD(\Omega)$ and $\MbDd$; $\sqrt{\C_1(\alpha(\cdot))}p(\cdot)$ is strongly continuous in $L^2(\Omega;\Mnn)$.
\end{thm}

We point out that we are exactly in the setting proposed in \cite{KazMar18} for the variational formulation. In particular, $\C_1(0)=0$ and the gradient damage term is taken in the $L^2$ norm. These constitutive choices give weaker regularity than in other coupled damage-plasticity models, which requires a careful mathematical treatment.

The degeneracy of $\C_1$ at the maximal damage state corresponds to complete damage. For this model, it means that in $\{\alpha=0\}$ the material satisfies Drucker-Prager perfect plasticity and $p$ is a bounded Radon measure, since it is controlled only linearly.
 Such assumption is then crucial to observe plastic shear bands, otherwise one would have a kinematical hardening even in $\{\alpha=0\}$ and no concentration of plastic strain. 
Furthermore, the $L^2$ damage gradient term is the one usually adopted in simulations and in mechanical models, e.g.\ \cite{CRASII, AMV, UllWamAleSamDegFra21}; the fact that the damage field is not continuous but only in $H^1$ is a source of analytical issues, since it in general multiplies some functions of the measure $p$, which is in duality with continuous fields.
We remark that it could be also interesting for future applications to deal with the case where the switching between perfect plasticity and plasticity with hardening is governed by the values of a Sobolev field, as in the present framework.

In previous treatments of quasistatic evolutions for coupled small-strain plasticity-damage models \cite{Cri16, Cri17, CriOrl18} and in many elasticity-damage models (see e.g.\ \cite{MieRou06, Tho13, KRZ13a, KneNeg17}) the damage is assumed incomplete: in those works the Hooke's tensor $\C$, depending on $\alpha$, is positive definite 
for any $\alpha\in [0,1]$. 
Here $\C$ is independent of $\alpha$ but the relevant mixed term including the damage variable degenerates. Restricting to elasticity-damage models, in \cite{BouMieRou09, MieRouZem07, Mie11, HeiKra15} some cases of complete damage are addressed, and also in the dynamic plasticity-damage model \cite{DavRouSte19} Hooke's tensor is only positive semidefinite. 
Moreover, in plasticity-damage models the $L^2$ regularisation in the damage gradient is at the moment not enough, unless having the presence of strain gradient \cite{Cri17} or further assumptions on elastic strain \cite{CriOrl20}: an $L^p$ regularisation is required for $p>n$ \cite{Cri16, DavRouSte19} or for $p=n$ \cite{CriOrl18}.

Following De Giorgi's Minimising Movement approach to quasistatic evolutions \cite{Amb95}, time-continuous evolutions are approximated by discrete-time ones, constructed by solving incremental minimisation problems. 
The interaction between damage in $H^1$ and plasticity plays a crucial role in the lower semicontinuity of the kinematical hardening term (see Theorem~\ref{thm:1505212332}), needed in incremental minimisation problems and both in passing to the limit  the discrete stability to get (qs1) and in the lower energy inequality in (qs2); moreover, it enters also in a continuity condition to get (qs1), see Theorem~\ref{thm:1905212052}.

Theorems~\ref{thm:1505212332} and \ref{thm:1905212052} are the main intermediate results needed to prove Theorem~\ref{thm:2305211857}: their proof strategies employ fine properties of Sobolev functions, derived from capacitary arguments, combined with tools from free discontinuity problems in linear elasticity. Indeed, we observe that even the definition of $\tQ$ is well posed since $p$ could concentrate on sets of dimension at least $n{-}1$, and the capacitary precise representative $\widetilde{\alpha}$ of $\alpha$ is well defined up to sets of dimension $s$ for any $s>n{-}2$.

In Theorem~\ref{thm:1505212332} we use a blow-up procedure, at every $x_0 \in \Omega$ with $\alpha(x_0)>0$. If the rescaled damage variables $\beta_k$ converge uniformly to $\alpha(x_0)>0$, then the rescaled plastic strains $q_k$ (for which $\int_{B_1} \C_1(\beta_k)q_k\colon q_k \,\mathrm{d}x$ is bounded) would be bounded in $L^2$ and the semicontinuity would be direct from Ioffe-Olech Theorem. Since $\beta_k$ converge weakly in $H^1$ to $\alpha(x_0)$, the convergence is uniform up to an exceptional set of small $s$-capacity, $s\in (1,2)$. The idea is to combine the fact that the exceptional set could be taken with small perimeter, following Lahti \cite{Lah17}, with the theory of $GSBD$ functions \cite{DM13}, namely those corresponding to displacements for linearly elastic materials outside a $(n{-}1)$-dimensional crack. We see the boundary of the exceptional set of small capacity as a discontinuity surface of the rescaled displacements $v_k$, that are regarded as functions equibounded in $GSBD$, since the absolutely continuous parts of their symmetric gradients are bounded in $L^2$ and their discontinuity sets have bounded surface measure. Using lower semicontinuity results in $GSBD$ we get lower semicontinuity outside the exceptional set.

Furthermore, to exclude that the measure $p$ concentrates on $\{\widetilde{\alpha}>0\}$, we resort to a slicing argument. Also here, the result would be direct under uniform convergence of damage $\alpha_k$ to $\alpha$, using that the sets $\{\alpha>\delta\}$ are open and well approximated by $\{\alpha_k>\delta\}$.
We stress the fact that in non associative plasticity regularity assumptions are needed even to solve the minimisation problems \cite{BabFraMor12} or to prove existence of weak evolutions \cite{FraMor17}: these are obtained through convolutions with fixed kernels. 

The $GSBD$-lower semicontinuity together with Lahti's capacitary estimate is applied also in Theorem~\ref{thm:1905212052}, without blow-up. 
We believe that this strategy could be useful in other problems in plasticity-damage and for other couplings in different contexts.

Another issue arising when following the framework of \cite{KazMar18} is that the triaxial test requires mixed Dirichlet-Neumann boundary conditions. These mixed conditions could be imposed trough a suitable notion of stress-plastic strain duality, well defined when the constraint set $K$ is bounded in the direction of deviatoric matrices, see e.g.\ \cite{Mora}. Conversely, this is a problem when dealing with geomaterials, so that fully Dirichlet boundary conditions are usually considered in mathematical works, see \cite[Introduction]{BabMor15}.
Here, besides safe load condition, we require a further assumption on the external loads; remarkably, this is satisfied in the case of triaxial test. These assumptions guarantee coercivity in $p$ in the incremental minimisation problems and permit us to obtain a weak formulation of evolution accounting for Neumann conditions in an integral form.
In the last part of the paper we show that, under regularity assumptions on the evolution, one recovers the differential properties in \cite{KazMar18}, to which we refer for the full set of conditions and for further details on their mechanical interpretation.

We notice that we follow \cite{KazMar18}, but with the same techniques one could treat the case of a 
$L^\gamma$ gradient damage regularisation, $\gamma>1$, in place of $L^2$. Moreover, taking $\C_1(0)$ positive definite or $\C_1(1)<+\infty$ simplifies the analysis. 

One could also repeat the same arguments for the existence of evolutions when adding a plastic strain gradient term of the type $\|\nabla p\|_{L^2}^2$ in the total energy (starting by minimising the functional in \eqref{1005211029} plus $\|\nabla p\|_{L^2}^2$); actually, this term simplifies the analysis since it improves the regularity of $p$.
Moreover, a strain gradient regularisation seems also suitable to improve time regularity.
In the current setting, quasistatic evolutions
are not proven to be absolutely continuous in time, as usual in coupled plasticity-damage models. Different notions of evolutions have been proposed to deal with time discontinuity, such as the one of $BV$ solutions \cite{MieRosSav13}, obtained trough a vanishing-viscosity technique à la Efendiev-Mielke \cite{EfeMie06} (see \cite{CriLaz16, CriRos21} for plasticity-damage models).  We plan to address these variants of the model 
in future works.


%
%

 \section{Preliminaries}\label{Sec:2}

 \medskip

\noindent {\bf{Mathematical preliminaries.}}
For every $x\in \Rn$ and $\varrho>0$, let $B_\varrho(x) \subset \Rn$ be the open ball with center $x$ and radius $\varrho$, and let $Q_\varrho(x) = x+(-\varrho, \varrho)^n$, $Q_\varrho^\pm(x) = Q_\varrho(x) \cap \{x \in \Rn \colon \pm x_1 >0\}$. For $\nu \in \Sn:=\{x \in \Rn \colon |x|=1\}$, we let also $Q_\varrho^\nu(x)$ the cube with ``center'' $x$, sidelength $\varrho$ and with a face in a plane orthogonal to $\nu$. We omit to write the dependence on $x$ when $x=0$.  (For $x$, $y\in \Rn$, we use the notation $x\cdot y$ for the scalar product and $|x|$ for the  Euclidean  norm.)   By ${\mathbb{M}^{n\times n}}$, ${\mathbb{M}^{n\times n}_{\rm sym}}$, and ${\mathbb{M}^{n\times n}_{\rm skew}}$ we denote the set of $n\times n$ matrices, symmetric matrices, and skew-symmetric matrices, respectively.   We write $\chi_E$ for the indicator function of any $E\subset \R^n$, which is 1 on $E$ and 0 otherwise.  If $E$ is a set of finite perimeter, we denote its essential boundary by $\partial^* E$, 
see \cite[Definition 3.60]{AFP}.   
 We indicate the minimum  and maximum  value between $a, b \in \R$ by  $a \wedge b$  and $a \vee b$, respectively. 
 
 The Lebesgue measure on $\Rn$ is denoted by $\mathcal{L}^n$ and the $(n{-}1)$-dimensional Hausdorff measure by $\mathcal{H}^{n-1}$.  The space of bounded $X$-valued Radon measures on $B$ is denoted by $M_b(B; X)$, for a locally compact subset $B$ of $\Rn$ and a finite dimensional Hilbert space $X$. The indication of the space $X$ is omitted when $X=\R$, and we write 
$M^+_b(B)$ for the subspace of positive measures of $M_b(B)$.  The space $M_b(B; X)$ is endowed with the norm $\| \mu \| _1 := \vert \mu \vert(B)$, where $\vert \mu \vert \in M_b(B)$ is the variation of the measure $\mu$, and it is identified with the dual of $C_0(B; X)$, the space of continuous functions $ \varphi \colon B \to X$ such that $\{\vert \varphi \vert \geq \varepsilon\}$ is compact for every $\varepsilon >0$, by the Riesz Representation Theorem (see, e.g., \cite[Theorem 6.19]{Rud}).
The weak$^*$ topology of $M_b(B; X)$ is defined using this duality.
 
 The space $L^1(B; X)$ of $X$-valued $\mathcal{L}^n$-integrable functions is regarded as a subspace of $M_b(B; X)$, with the induced norm. The $L^p$ norm, $1 \leq p \leq \infty$ is denoted by $\| \cdot \|_p$, while the brackets $\langle \cdot, \cdot \rangle$ denote the duality product between conjugate $L^p$ spaces.
 
The space 
 $\Mnn$
 is endowed with the Euclidean scalar product $\xi \colon \eta:= \sum_{ij} \xi_{ij} \eta_{ij}$ and with the corresponding Euclidean norm $\vert\xi \vert:=(\xi \colon \xi)^{1/2}$.
The symbol for the space of trace free matrices in $\Mnn$ is $\MD$.
For every $\xi\in\Mnn$ the orthogonal projection 
of $\xi$ on $\R I$ is $\frac{1}{n} \text{tr }(\xi)I$. Therefore the 
orthogonal projection on $\MD$, called the {\it deviator\/} of $\xi$, is
\[\xi_D:=\xi-\frac{1}{n}(\tr \,\xi) I\,.\]
\CCC We denote also
\[
\xi_m:=\frac{1}{n}(\tr \,\xi)\,.
\] \EEE
The \emph{symmetrized tensor product $a\odot b$} of two vectors $a,\,b \in \Rn$ is the symmetric matrix with entries $(a_ib_j+a_jb_i)/2$. If $X_1,\,X_2$ are Banach spaces, $Lin(X_1;X_2)$ is the space of linear operators from $X_1$ into $X_2$, endowed with the usual operator norm.

For every $u \in L^1(U; \Rn)$, with $U$ open in $\Rn$, let $Eu$ be the $\Mnn$-valued distribution on $U$ whose components are defined by $E_{ij}u=\frac{1}{2}(D_ju_i+D_iu_j)$. The space $BD(U)$ of functions with \emph{bounded deformation} is the space of all $u \in L^1(U; \Rn)$ such that $Eu \in \Mb(U; \Mnn)$. It is easy to see that $BD(U)$ is a Banach space with the norm $\|u\|_1 + \|Eu\|_1$. It is possible to prove that $BD(U)$ is the dual of a normed space (see \cite {TemStr} and \cite{MSC}), and this defines the weak$^*$ topology of $BD(U)$.
 A sequence $u_k$ converges to $u$ weakly$^*$ in $BD(U)$ if and only if $u_k \rightarrow u$ strongly in $L^1(U;\Rn)$ and $Eu_k \wto Eu$ weakly$^*$ in $\Mb(U; \Mnn)$. 
If $U$ is a bounded open set with Lipschitz boundary, for every function $u \in BD(U)$ the trace of $u$ on $\partial U$ belongs to $L^1( \partial U; \Rn)$.  It will always be denoted by the same symbol $u$.
If $u_k$, $u\in BD(U)$, $u_k\to u$ strongly in $L^1(U;\Rn)$, and $\|Eu_k\|_1 \to \|Eu\|_1$, then $u_k\to u$ strongly in $L^1(\partial U;\Rn)$ (see \cite[Chapter~II, Theorem~3.1]{Tem}). Moreover  (see \cite[Proposition~2.4 and Remark~2.5]{Tem}), there exists a constant $C>0$, 
depending on $U$, such that
\begin{equation}
\label{seminorm}
\|u\|_{1,U} \le C\, \|u\|_{1,\partial U}+ 
C\, \|Eu\|_{1,U}\,,
\end{equation} 
$\|\cdot \|_{p,B}$ being the $\emph{L}^p$ norm of a function with  domain a Borel set $B$. 

It is well known (see \cite{ACD97, Tem}) that for $v\in BD(U)$  the jump set  $J_v$ is countably $(\hn, n{-}1)$ rectifiable, and that
\begin{equation*}
\mathrm{E}v=\mathrm{E}^a v+ \mathrm{E}^c v + \mathrm{E}^j v,
\end{equation*}
where $\mathrm{E}^a v$ is absolutely continuous with respect to $\Ln$, $\mathrm{E}^c v$ is singular with respect to $\Ln$ and such that $|\mathrm{E}^c v|(B)=0$ if $\hn(B)<\infty$, while $\mathrm{E}^j v$ is concentrated on $J_v$. The density of $\mathrm{E}^a v$ with respect to $\Ln$ is denoted by $\E(v)$.

The space $SBD(U)$ is the subspace of all functions $v\in BD(U)$ such that $\mathrm{E}^c v=0$. For $p\in (1,\infty)$, we define
$SBD^p(U):=\{v\in SBD(U)\colon \E(v)\in L^p(U;\Mnn),\, \hn(J_v)<\infty\}$.
For a complete treatment of $BD$ and $SBD$ functions, we refer to 
to  \cite{ACD97,  Tem}.

The spaces $GBD(U)$ of \emph{generalized functions of bounded deformation} and $GSBD(U)\subset GBD(U)$  of \emph{generalized special functions of bounded deformation} have been introduced in \cite{DM13} (cf.\ \cite[Definitions~4.1 and 4.2]{DM13}). We recall that
every $v\in GBD(U)$ has an \emph{approximate symmetric gradient} $\E(v)\in L^1(U;\Mnn)$ and an \emph{approximate jump set} $J_v$ which is still countably $(\hn,n{-}1)$ rectifiable  (cf.~\cite[Theorem~9.1,  Theorem~6.2]{DM13}).

The notation for $\E(v)$ and $J_v$, which is the same as that one in the $SBD$ case, is consistent: in fact, if $v$  lies in  $SBD(U)$, the objects coincide (up to  negligible sets of points with respect to $\Ln$ and $\hn$, respectively). 
For $1 < p < \infty$, the   space $GSBD^p(U)$ is  given by  
\begin{equation*}
GSBD^p(U):=\{v\in GSBD(U)\colon \E(v)\in L^p(U;\Mnn),\, \hn(J_v)<\infty\}.
\end{equation*}
The theory of $GSBD$ functions has been developed with many contributions in recent years, we refer e.g.\ to \cite[Section~3]{FriPerSol21} for a general picture; in this paper we employ compactness in $GSBD$ from \cite{CC21JEMS}, see also \cite[Theorem~11.3]{DM13}.

 \medskip

\noindent {\bf{The reference configuration.}} Throughout the paper the \emph{reference configuration} $\Omega$ is a \emph{bounded connected open set} in $\Rn, \, n \geq 2$, with \emph{Lipschitz boundary}. 
We assume that Dirichlet and possibly Neumann boundary conditions are imposed, corresponding to two subsets $\dod \neq \emptyset$ and $\don$ of $\dom$ satisfying
\begin{taggedsubequations}{BC}\label{do}
\begin{equation}\label{2405211931}
\dom=\dod \cup \don,\quad \dod \cap \don=\emptyset\,,
\end{equation}
$\dod$ being the part of $\dom$ where the displacement is prescribed, while  traction forces are applied on $\don$.
Here $\dod$ and $\don$ are open (in the relative topology), with the same boundary $\Gamma$ such that 
\begin{equation}
\mathcal{H}^{n-2}(\Gamma)<+\infty\,.
\end{equation}
\end{taggedsubequations} \EEE
Along the paper we deal with the general case, intending that the case of pure Dirichlet boundary conditions corresponds to formally consider $\dod=\dom$ in \eqref{2405211931}.
 \medskip


\noindent {\bf{The elastic and plastic strains.}}  In our problem $u\in BD(\Omega;\Rn)$ represents the {\it displacement\/} of an elasto-plastic body and $Eu$ is the corresponding linearised {\it strain\/}. We now introduce the coupled elasto-plastic damage model. As for modelling plasticity, we follow \cite{Mora} and use the corresponding notations.

Given a displacement $u\in BD(\Omega; \Rn)$ and a 
boundary datum $w\in H^1(\Rn;\Rn)$, the {\it elastic\/} and {\it plastic strains\/} $e\in 
L^2(\Omega;\Mnn)$ and $p\in \MbDd$ satisfy the equations (weak kinematic compatibility conditions)

\begin{equation}\label{200}
\EE u=e+p \quad \hbox{a.e.\ in }\Omega\,, 
\end{equation}

Given $w\in H^1(\Rn;\Rn)$, the {\it set of admissible displacements and strains\/} for the boundary datum $w$ on $\dod$ is defined, with the same meaning and notation of \cite{Mora}, as 
\begin{equation}\label{1605210915}
\begin{split}
\A(w):=\{(u,e,p) &\in BD(\Omega; \Rn) \times L^2(\Omega;\Mnn) \times \MbDd \colon \,\eqref{200} \,\text{holds} \\ & \text{and } p=(w-u) \odot \nu \hn \mres \partial_D \Omega\, \}\,.
\end{split}
\end{equation}
Differently from \cite{Mora}, here $p$ has not null trace in general, so that $w-u$ could also have a non null normal component on $\dod$. If $\don=\emptyset$, \eqref{1605210915} reduces to the space of admissible configurations defined in \cite[Section~6]{BabMor15}.

 \medskip

\noindent {\bf{The damage variable and the associated dissipation.}} In addition to $u,\, e$, and $p$, we consider an internal variable $ \alpha \colon \Omega \to [0,1]$, which represents the damage state of the body.
At a given point $x \in \Omega$, as $\alpha(x)$ decreases from $1$ to $0$, the material point $x$ passes from a sound state to a fully damaged one.
\CCC During the evolution, the damage variable is forced to decrease. \EEE


In the total energy we consider a term which accounts for the energy dissipated by the body during the damage process. The \emph{dissipation term} is a functional
\begin{equation}\tag{D}
D\colon L^1(\Omega) \to \R^+\cup\{0\} \,\text{ strongly continuous}\,.\label{eq:D}
\end{equation}
We do not require that $D$ is nonincreasing or positively one-homogeneous, because it is not needed to prove our result.
However, such assumptions would be natural, since $D$ represents a dissipation. \CCC In \cite{KazMar18} $D$ is linear in $\alpha$. \EEE

\CCC The total energy includes also a damage gradient  term $\| \nabla \alpha \|_{2}^{2}$ on the damage variable, cf. \eqref{1005211029}.
In particular, whenever the enegy is finite the damage variable will be in $H^1(\Omega)$. The introduction of damage gradient term is recurrent in the study of evolution for damage models trough variational methods (see for instance \cite{CRASII} and \cite[Section~4]{AMV}).  \EEE
\medskip

\paragraph*{\textit{Capacity.}} For the notion of capacity we refer, e.g., to \cite{EvaGar, HeiKilMar}. We recall here the definition and some properties.  

Let $1 \leq \gamma < +\infty$ and let $\Omega$ be a bounded, open subset of $\Rn$. For every subset $B \subset \Omega$, the \emph{$\gamma$-capacity} of $E$ in $\Omega$ is defined by
\[
\Cap_\gamma(E,\Omega) := \inf\Big\{ \int\limits_\Omega |\nabla v|^\gamma \dx  \colon v \in W^{1,\gamma}_0(\Omega), \ v \geq 1 \text{ a.e.\ in a neighbourhood of } E \Big\} \, .
\]
A set $E \subset \Omega$ has \emph{$\gamma$-capacity zero} if $\Cap_\gamma(E, \Omega) = 0$ (actually, the definition does not depend on the open set $\Omega$ containing $E$). A property is said to hold \emph{$\Cap_\gamma$-quasi everywhere} (abbreviated as $\Cap_\gamma$-q.e.) if it holds for a set of $\gamma$-capacity zero.

If $1 < \gamma \leq n$ and $E$ has $\gamma$-capacity zero, then $\HH^s(E) = 0$ for every $s > n-\gamma$.

A function $\alpha \colon \Omega \to \R$ is $\Cap_\gamma$-{\em quasicontinuous} if for every $\varepsilon > 0$ there exists a set $E_\varepsilon \subset \Omega$ with $\Cap_\gamma(E_\varepsilon, \Omega) < \varepsilon$ such that the restriction $\alpha|_{\Omega \sm E_\varepsilon}$ is continuous. Note that if $\gamma > n$, a function $\alpha$ is $\Cap_\gamma$-quasicontinous if and only if it is continuous.  

Every function $\alpha \in W^{1,\gamma}(\Omega)$ admits a $\Cap_\gamma$-{\em quasicontinuous representative} $\tilde \alpha$, i.e., a $\Cap_\gamma$-quasicontinuous function $\tilde{\alpha}$ such that $\tilde \alpha = \alpha$ $\Ln$-a.e.\  in $\Omega$. The $\Cap_\gamma$-quasicontinuous representative is essentially unique, that is, if $\tilde \beta$ is another $\Cap_\gamma$-quasicontinuous representative of $\alpha$, then $\tilde \beta = \tilde{\alpha}$ $\Cap_\gamma$-q.e.\  in~$\Omega$.  It satisfies (see \cite[Theorem~4.8.1]{EvaGar})
\begin{equation} \label{eq:precise representative}
    \lim_{\rho \to 0} \frac{1}{|B_\rho(x_0)|} \int\limits_{B_\rho(x_0)} \! |\alpha(x) - \tilde \alpha(x_0)| \dx = 0 \, \quad \text{for $\Cap_\gamma$-q.e.\ $x_0 \in \Omega$} \, .
\end{equation}
Recalling \cite[Theorem~7]{Eva90}, if $\alpha_k \wto \alpha$ in $W^{1,\gamma}(\Omega)$, up to passing to a subsequence we have that for every $\wt\gamma \in [1,\gamma)$ and every $\varepsilon>0$ there exists a relatively open set $E_\varepsilon\subset \Omega$ such that
\begin{equation}\label{1205211105'}
\wt\alpha_k \to \wt\alpha \quad\text{uniformly on } \Omega \sm E_\varepsilon,\qquad \mathrm{Cap}_{\wt\gamma}(E_\varepsilon, \Omega) < \varepsilon.
\end{equation}
 \medskip

\noindent {\bf{The stored energy.}} The stored energy is the sum of two contributions: besides the classical expression of the elastic energy depending only on the elastic strain through Hooke's tensor, the model devised in \cite{KazMar18} displays a further term depending both on the damage variable and the plastic strain.

 \medskip

\paragraph{\textit{The elastic energy of the sound material.}} \CCC For every $e \in L^2(\Omega, \Mnn)$, when the damage variable assumes constant value 0,
the {\it elastic energy\/} is given by
\begin{equation*}
\QQ(e):=\frac{1}{2}\int_{\Omega} \C e(x) : e(x) \dx = \frac{1}{2}\langle \C e, e\rangle _{L^2(\Omega; \Mnn)},
\end{equation*}
with $\C$ the Hooke's tensor, or any fourth order tensor satisfying the symmetry conditions
\begin{taggedsubequations}{$\C$}\label{eq:Celast}
\begin{equation}\label{0402210844}
\C_{ijkl}=\C_{klij}=\C_{jikl}\quad\text{for all }i,j,k,l \in \{1,\dots, n\}
\end{equation} 
and the coercivity-continuity assumption
\EEE
\begin{equation}\label{0302211656}
  \gamma_1 |\xi |^2  \leq \C \xi : \xi \leq \gamma_2 |\xi |^2 \quad \text{for every} \, \xi \in \Mnn,
\end{equation}
\end{taggedsubequations}
 where $\gamma_1,\,\gamma_2$ are positive constants.
 In particular, this implies
 \begin{equation}
 |\C \xi| \leq 2 \gamma_2 |\xi|. \label{C4}
 \end{equation}
 
 \medskip

\noindent {\textit{The kinematical hardening term.} } 
The latter term in the elastic energy has the form
\begin{equation}\label{1205211245}
\int_\Omega \C_1(\alpha(x)) p(x) \colon p(x) \dx, \quad \BBB p=0 \text{ in }\{\alpha=1\}\EEE
\end{equation}
with $\C_1$ a fourth order tensor satisfying 
\begin{taggedsubequations}{$\C_1$}
\label{eq:C}
\begin{gather}
\BBB \C_1 \in C([0,1); Lin(\Mnn;\Mnn)),\label{C1} \EEE
 \\ 
  \gamma(\alpha) |\xi|^2  \leq \C_1(\alpha) \xi : \xi  \quad \text{for every} \, \alpha \in [0,1], \xi \in \Mnn, \label{C2}\\
   \C_1(\alpha)=0 \quad \text{if and only if }\alpha=0, \label{C2'}
\\
\alpha  \mapsto \C_1(\alpha) \xi : \xi \quad \text{is nondecreasing for every}\,\xi \in \Mnn, \label{C3} 
\\
\BBB \lim_{\alpha\to 1^-} \C_1(\alpha)\xi\cdot \xi=+\infty \quad\text{uniformly w.r.t.\ }\xi \in \Mnn,  \EEE \label{C5} 
\end{gather}
so that we may assume
\begin{equation}\label{propgamma}
\gamma\colon [0,1]\to \R^+ \text{ nondecreasing with }\gamma(\alpha)=0 \text{ if and only if }\alpha=0.
\end{equation}
\end{taggedsubequations}
\CCC This represents the energy blocked by the contact with friction of the lips of the cracks, and it results in a hardening behaviour. Its derivation is based on a micro-mechanical approach. 
Observe that $\C_1$ is not coercive uniformly with respect to $\alpha\in [0,1]$ since $\C_1(0)=0$, in contrast to the assumptions  on $\C(\alpha)$ in the incomplete damage models.

The functional in \eqref{1205211245} is well defined if $\alpha \in L^1(\Omega; [0,1])$ and $p\in \Lnn$. Since in our model the natural space for $p$ is 
$\MbDd$, we extend the definition to $\alpha \in H^1(\Omega;[0,1])$ and to those $p\in\MbDd$ for which there exist $u\in BD(\Omega)$, $e \in L^2(\Omega)$, $w \in H^1(\Omega)$ satisfying \eqref{200}, by setting
\BBB \begin{equation*}
\tQ(\alpha,p):=
\begin{dcases}
\int_{\Omega\sm\{\alpha=1\}} \hspace{-3em}\C_1(\alpha(x)) p(x) \colon p(x) \dx \quad\text{if }|p^s|(\{ x \in \Omega\cup \dod\colon\wt\alpha(x)>0\})=0, \, p \,\chi_{\{\alpha=1\}}=0,\\
+\infty \quad\text{otherwise,}
\end{dcases}
\end{equation*} \EEE
where $\wt \alpha$ is the quasicontinuous representative of $\alpha \in H^1(\Omega)$ and $\wt\alpha$ is defined in $\dom$ through the trace of $\alpha$ (we could consider any $H^1$ extension of $\alpha$ in an open set $\Omega' \supset \ol\Omega$).
Notice that $\wt \alpha$ is well defined up to a negligible set with respect to $\mathrm{Cap}_2$; by \eqref{200} and properties of $BD$ functions, $p$ may concentrate up to sets of nonnegative $\hn$ measure, and not on $\mathrm{Cap}_\delta$-negligible sets, for any $\delta>1$. Then it is meaningful to write $|p^s|(\{\wt\alpha>0\})$.
\BBB Observe also that $p \in L^1(\{\wt\alpha>0\})$ has to be 0 in $\{\alpha=1\}$. This formally corresponds to take $\C_1(1)=+\infty$, as requested in \cite[Section~3.3]{KazMar18}. \EEE
\EEE
 \medskip


\noindent {\bf{The constraint set and its support function.}} 
%
\CCC Let $K$ be a closed convex set in $\Mnn$ containing a ball of radius $r_H>0$, that is
\begin{equation}\label{0302211716}\tag{K}
\{ \sigma \in \Mnn \colon |\sigma|\leq r_H\} \subset K\,.
\end{equation} 
 Let us consider the \CCC support \EEE function $H\colon \Mnn \to \R^+\cup\{0\}$ 
\begin{equation*}
H(\xi):=\sup_{\sigma\in K} \sigma \colon \xi\,, 
\end{equation*} 
which, by \eqref{0302211716}, satisfies $H(p) \geq r_H |p|$ for all $p\in \Mnn$.
We remark that in the applications considered in \cite{KazMar18}, namely the Drucker-Prager criterion (or some variants like Mohr-Coulomb or Hoek-Brown criteria) with kinematical hardening coupled with damage, $K$ is unbounded in the direction of negative hydrostatic matrices: it has the form
\begin{equation*}
K:=\{\sigma\in \Mnn \colon \tau \sigma_m + \kappa(\sigma_D) -k\leq 0\}\,,
\end{equation*}
with $\kappa\colon \MD \to [0,+\infty)$ convex and positively 1-homogeneous, $\kappa(0)=0$, and $\tau>0$, $k>0$ two constants.
\medskip

\noindent {\bf{The plastic potential.}} 
Basing on the theory of convex functions of measures developed in \cite{GofSer}, we define the non-negative Borel measure
\begin{equation*}
\textbf{H}(p):= \int_{\Omega\cup \dod} H\biggl(\frac{dp}{d|p|}(x) \biggr)\,d|p|(x)\,,\quad\text{for }B\subset \Omega\cup \dod \text{ Borel.}
\end{equation*} 
If $\textbf{H}(p)$ has finite mass, namely it is a bounded Radon measure, we introduce the
\emph{plastic potential} $\HH \colon \MbDd \to \R$ by
\begin{equation}\label{HH}
\HH(p):=\textbf{H}(p)(\Omega\cup \dod)\,.
\end{equation}  
In this case, the results in \cite{DemTem84, DemTem86} (see in particular \cite[Theorem~2.1]{DemTem86}) give that $\textbf{H}$ is expressed through duality formulas. For a bounded smooth open set $\Omega'$ such that $\Omega\subset \Omega'$ and $\Omega' \cap \dom=\dod$, extending any $p$ by 0 on $\Omega'\sm \Omega$, it holds that   \CCC
\begin{equation}\label{2607210847} 
\int_{\Omega'}\varphi \,\mathrm{d}[\textbf{H}(p)]=\sup_{\sigma\in C_c^\infty(\Omega';K)}\int_{\Omega'} \varphi\, \sigma \colon \mathrm{d}p =\sup_{\sigma\in C_c(\Omega';K)}\int_{\Omega'} \varphi\, \sigma \colon \mathrm{d}p =\sup_{\sigma\in L^1(\Omega',|p|+\LL^n;K)}\int_{\Omega'} \varphi\, \sigma \colon \mathrm{d}p 
\end{equation}
for any $\varphi \in C_c(\Omega')$, $\varphi\geq 0$, and
\begin{equation}\label{2607210848} 
\HH(p)=\sup_{\sigma\in C_c^\infty(\Omega';K)} \int_{\Omega'} \sigma \colon \mathrm{d}p=\sup_{\sigma\in C_c(\Omega';K)} \int_{\Omega'} \sigma \colon \mathrm{d}p=\sup_{\sigma\in L^1(\Omega',|p|+\LL^n;K)}\int_{\Omega'} \sigma \colon \mathrm{d}p,
\end{equation}
where $L^1(\Omega',|p|+\LL^n;K)$ denotes the space of integrable functions with respect to $|p|+\LL^n$ with values in $K$.
  \EEE
Moreover, Reshetnyak Theorem (see \cite[Theorem 2.38]{AFP}) implies that $\HH$ is sequentially weakly$^*$-lower semicontinuous in $\MbDd$.
  \medskip

\noindent {\bf{The plastic dissipation.}}
We introduce now a term which represents the plastic dissipation in a given time interval.

A function $p\colon [0,T]\to \MbDd$ will be regarded as a function defined on the time interval $[0,T]$ 
with values in the dual of the separable Banach space $C(\Omega\cup \dod;\Mnn)$, that can be identified with the space of functions in $C(\ol\Omega;\Mnn)$ vanishing on $\overline{\don}$.
For every $s,t\in[0,T]$ with 
$s\leq t$ the total variation of $p$ on $[s,t]$ is defined by
\[
\V(p;s,t)=\sup\biggl\{\sum_{j=1}^N\|p(t_j)-p(t_{j-1})\|_1\,\Big|\, s=t_0< 
t_1< \dots< t_N=t, \, N\in\N
\biggr\}\,.
\]
\CCC For every partition $\PP$ of $[s,t]$, namely $\PP:=\{t_i\}_{0\leq i \leq N}$ with $s=t_0< 
t_1< \dots < t_N=t$, we define
\[ \V^\PP_\HH(p;s,t):=\sum_{i=1}^N \HH(p(t_i)-p(t_{i-1}))\,. \]
The \emph{$\HH$-variation of $p$} on $[s,t]$
is denoted by $\V_\HH(p;s,t)$ and is defined through
\begin{equation}\label{diss}
\begin{split}
\V_\HH(p;s,t):=\sup&\biggl\{ \sum_{j=1}^N \HH(p(t_j)-p(t_{j-1}))\,\Big|\, 
s=t_0< t_1< \dots< t_N=t, \,
N\in\N \biggr\}\\=\sup &\bigl\{\V_\HH^{\PP}(p;s,t)|\,\PP \text{ partition of $[s,t]$}\bigr\}.
\end{split}
\end{equation}
We recall that this notion has been introduced in \cite[Appendix]{Mora}. 
\EEE

 \medskip

\noindent {\bf{The prescribed boundary displacement and the external loading.}} 
We assume that the prescribed boundary displacement $w$ depends on time and satisfies the regularity assumption 
\begin{taggedsubequations}{Load}\label{eq:data}
\begin{equation}
w \in AC(\ot; H^1(\Rn;\Rn)), \label{w1}
\end{equation} 
so that the 
time derivative
$t\mapsto\dot w(t)$ belongs to $L^1([0,T]; H^1(\Rn;\Rn))$ and 
its strain $t\mapsto E\dot w(t)$ belongs to
$L^1([0,T];L^2(\Rn;\Mnn))$. The prescribed boundary value will be 
the trace on $\dod$ of $w$ (again denoted by $w$).
\CCC As for the volume force $f$, we assume that
\begin{equation}\label{0302212216}
f \in AC(\ot; L^n(\Rn;\Rn)).
\end{equation}
In the presence of Neumann boundary conditions, the traction is a function
\begin{equation}\label{1707212200}
g \in AC(\ot; L^\infty(\don)).
\end{equation}
We define the \emph{total load} $\LL \in AC(\ot; BD(\Omega)')$ by
\begin{equation}\label{0905211246}
\langle \LL(t) , v \rangle := \int_\Omega f(t) \cdot v \dx + \int_{\don} g(t)\cdot  v \dH,
\end{equation}
for every $v \in BD(\Omega)$. Notice that $v \in L^{1^*}(\Omega;\Rn)$ and $v \in L^1(\dom; \Rn)$, by embedding properties of $BD$ functions.
If we assume only Dirichlet boundary conditions, \eqref{0905211246} reads as $\langle \LL(t) , v \rangle := \int_\Omega f(t) \cdot v \dx$.
\EEE 
 For the main properties of absolutely continuous functions with values in reflexive Banach spaces we refer to \cite[Appendix]{Bre}.
 
We assume the following \emph{safe load} condition: there exists $\varrho \in AC(\ot; L^n(\Omega;\Rn))$ and a constant $\tau_0 >0$ such that for every $t\in \ot$
\begin{equation}\label{0905211302}
\begin{dcases}
-\diver \varrho(t)=f(t)\quad\text{a.e.\ in }\Omega\,,\qquad [\varrho(t) \nu]=g(t) \quad\hn\text{-a.e.\ in }\don,\\
 \varrho(t) + \tau \in K \quad\text{for every }|\tau|\leq \tau_0.
\end{dcases}
\end{equation}
This is a standard assumption to study existence results for quasistatic evolutions in plasticity. It results in coercivity for the plastic strain in the incremental minimisation problem.
In the condition $[\varrho(t) \nu]=g(t) \quad\hn\text{-a.e.\ in }\don$, $[\varrho(t) \nu]$ is the distribution on $\dom$ defined by
\begin{equation*}
\langle [\varrho(t) \nu],\psi\rangle:= \int_\Omega \diver \varrho(t) \psi \dx + \int_\Omega \varrho(t) \EE \psi \dx\,.
\end{equation*}
Notice that when $\diver \varrho\in L^q(\Omega;\Rn)$, $\varrho \in L^q(\Omega;\Mnn)$, then $[\varrho \nu]$ is in the dual space of $W^{q', 1-q'}(\dom)$. 

When dealing with Neumann boundary conditions in the present case, where the constraint set is unbounded in the deviatoric matrices, we enforce a further condition. We have to require that
\begin{equation}\label{2607210753}
|\varrho(t)\colon p| \leq C_\varrho|p|\qquad\text{in }\MbDd,
\end{equation}
\end{taggedsubequations}
for $C_\varrho>0$ independent of $t\in [0,T]$. \EEE
It is not clear in general how to determine the fields $\varrho$ satisfying the safe load condition. We notice that in the case of null forces the safe load condition is satisfied for $\varrho\equiv 0$, 
and when only a constant pressure $\mathrm{pr}(t)$ is applied on $\don$ with some Dirichlet conditions on $\dod$ (with null volume loadings), as in the case of the triaxial test, $\varrho(t)$ could be taken constant for every $t$, so that \eqref{2607210753} is satisfied if $\mathrm{pr}(t)$ is bounded uniformly in time. \EEE 
In the case of Dirichlet boundary conditions, we drop the condition on $[\varrho(t) \nu]$ in \eqref{0905211302}, besides \eqref{2607210753}.

By \eqref{0905211302} and \eqref{2607210753}, 
we obtain
 for every $t\in \ot$ and $(u,e,p) \in \A(w(t))$ the following integration by parts formula 
\begin{equation}\label{1605211221}
\langle \mathcal{L}(t), u \rangle= \langle \varrho(t), e - \EE w(t)  \rangle_{L^2} + \int_{\Omega\cup \dod}\varrho(t)\colon \mathrm{d} p
+ \langle \mathcal{L}(t), w(t) \rangle.
\end{equation}


\medskip

\noindent {\bf{The initial data.}}
The initial configuration for the evolution $\alpha_0$, $u_0$, $e_0$, $p_0$ is such that $\alpha_0\in H^1(\Omega;[0,1])$, $(u_0, e_0, p_0) \in \A(w(0))$ and
\begin{equation}\label{2305211859}\tag{IC}
\begin{split}
\QQ(e_0) &+ D(\alpha_0) +\|\nabla \alpha_0\|_2^2 +\tQ(\alpha_0, p_0) -\langle \LL(0), u_0 \rangle  \\& \leq
\QQ(\eta)+D(\beta)+\| \nabla \beta\|_2^2 + \tQ(\beta, q)  +\HH(q-p_0) - \langle \LL_k(t), v \rangle
\end{split}
\end{equation}
for every $\beta \leq \alpha_0$, $\beta\geq 0$, $(v,\eta,q)\in \A(w(0))$.

  \section{The minimisation problem and the discrete time solutions}
  
In this section we study the minimisation problems entering the time discrete approximation of quasistatic evolutions for the present model. This approximation follows the general scheme of minimising movements.

At each incremental time step, the updated approximate solution is obtained by minimising, among the admissible configurations for the updated external loading, the sum of the internal energy terms, of the loading, and of the dissipation from the approximate solution at the previous time step.

More precisely, for every $k\in \N$ a sequence of subdivisions $(\tki)_{0\leq i\leq k}$ of the 
interval $[0,T]$ is introduced, with
\begin{eqnarray*}
& 0=t_k^0<t_k^1<\dots<t_k^{k-1}<t_k^{k}=T\,,\\
&\displaystyle
\lim_{k\to\infty}\,
\max_{1\le i\le k} (\tki-\tkim)= 0\,.
\end{eqnarray*}

Starting from $(\alpha_k^0, u_k^0,e_k^0, p_k^0):=(\alpha_0, u_0,e_0,p_0) \in \WWi \times \A(w(0))$, given $\wki:=w(\tki)$, $\LL_k^i:=\LL(\tki)$, for $i=1,\ldots,k$ we 
define $(\aki, \uki,e_k^i,\pki)$ as a solution to the incremental problem
\begin{equation}\label{1005211029}
\min_{0\leq\alpha\leq \alpha_k^{i-1},\, (u,e,p) \in \A(\wki)} \Big\{ \QQ(e)+D(\alpha)+\da + \tQ(\alpha,p) + \HH(p-p_k^{i-1}) - \langle \LL_k^i, u \rangle  \Big\}.
\end{equation}
We now discuss existence of solutions to the problem above.

First we notice that, assuming the problem at time-step $i-1$ admits minimisers, the minimising functional (at time-step $i$) is finite on $(\alpha_k^{i-1}, u_k^{i-1}+ \wki-w_k^{i-1}, e_k^{i-1} + E(\wki-w_k^{i-1}), p_k^{i-1})$, so the infimum is not $+\infty$. 

Moreover, we prove an estimate which will enforce coerciveness in $p$, see also Proposition~\ref{prop:1505212300} later on. For any fixed $\varepsilon>0$, letting $E_\varepsilon\subset \Omega'$ such that $|p|(\Omega'\sm E_\varepsilon)<\varepsilon$ and $\psi_\varepsilon \in C_c^\infty(\R^n)$ a cutoff function such that $\psi_\varepsilon=1$ on $E_\varepsilon$ and $\psi_\varepsilon=0$ on $\Rn\sm \Omega'$, by \eqref{0905211302} and \eqref{2607210753}, recalling \eqref{2607210848},  we have that
%
\begin{equation}\label{1505212325}
\begin{split}
\HH(p)  - \int_{\Omega\cup \dod}\varrho(t)\colon \mathrm{d} p
\geq 
 & \sup \Big\{ \int_{\Omega\cup \dod}  \big(\sigma - \psi_\varepsilon\varrho(t)\big) \colon \mathrm{d}p \ \colon \sigma \in C_c(\Omega';K) \Big\} \\& - \int_{\Omega\cup \dod} (1-\psi_\varepsilon) \varrho(t)\colon \mathrm{d}p 
\\& \geq \sup \Big\{\int_{\Omega\cup \dod} \widehat{\sigma}\colon \mathrm{d}p \colon \widehat{\sigma}\in C_c(\Omega'; B_{\tau_0}(0)) \Big\}- C_\varrho \, \varepsilon
\\& 
 \geq \tau_0 |p|(\Omega\cup \dod) - C_\varrho \, \varepsilon,
\end{split}
\end{equation}
using density arguments and the fact that $\psi \, \varrho(t) + \tau \in K$ for every $\psi \in [0,1]$, $t\in [0,T]$, and $|\tau|\leq \tau_0$, which follows from the convexity of $K_{\tau_0}:=\{\xi \in \Mnn \colon \xi + \tau \in K \text{ for all }|\tau|<\tau_0\}$ and since 0 is an interior point of $K_{\tau_0}$ (in view of \eqref{0905211302} and of the assumptions on $K$). 
Notice that we may exploit the finiteness of $\HH(p)$ as in \cite[Proposition~6.1]{BabMor15}, being the force term bounded and since we work on minimising sequences.
The semicontinuity of the plastic dissipation plus the external loading is obtained arguing as in \cite[Theorem~3.3]{Mora}, see also Proposition~\ref{prop:1505212300}.
At this stage, the only term whose semicontinuity is not directly ensured is $\tQ(\alpha,p)$. This is proven in the following result.

\begin{thm}\label{thm:1505212332}
Let $\alpha_k \wto \alpha$ in $H^1(\Omega;[0,1])$,
$(u_k, e_k, p_k) \in \A(w_k)$, $(u,e,p) \in \A(w)$ with $w_k\to w$ in $H^1(\Omega;\Rn)$, $p_k \wtos p$ in $\MbDd$, $e_k \wto e$ in $\Lnn$. 
Then
\begin{equation}\label{1105210945}
\tQ(\alpha,p) \leq \liminf_{k\to +\infty} \tQ(\alpha_k, p_k).
\end{equation}
\end{thm}
\begin{proof}
First, we notice that it is not restrictive to argue in an open set in place of $\Omega\cup \dod$. In fact, given an open set $\Omega' \supset \Omega \cup \dod$ with $\Omega'\cap \dom=\dod$, we may define 
\begin{equation*}
\check{Q}(\check{\beta},\check{q})=
\begin{dcases}
\int_{\Omega'\sm \{ \check{\beta}=1 \}} \C_1(\check{\beta}) \check{q} \colon \check{q}\dx,\quad &\text{if }|\check{q}^s|(\{\check{\beta}>0\})=0,\ \check{q}\, \chi_{\{\check{\beta}=1\}}=0,\\
+\infty \quad&\text{otherwise,}
\end{dcases}
\end{equation*}
for $\check{\beta}$ an extension of $\beta$ to $H^1(\Omega')$ \BBB (we can always assume $\check{\beta} \in H^1(\Omega';[0,1])$) \EEE chosen in such a way to ensure the weak-$H^1$ convergence in passing from $\Omega$ to $\Omega'$, and 
\begin{equation*}
\check{q}:=
\begin{dcases}
q \quad&\text{in } \Omega\cup \dod,\\
\BBB 0 \EEE \quad&\text{in }\Omega' \sm (\Omega\cup \dod).
\end{dcases}
\end{equation*}
\BBB Observe that $\check{q}$ is such that $(\check{v}, \check{\eta}, \check{q})\in \A(w)$, for $\check{v}$ the extension of $v$ with $w$ in $\Omega'\sm \Omega$ and $\check{\eta}$ the extension of $\eta$ with $\EE w$ in $\Omega'\sm \Omega$. 
Thus $\check{Q}(\check{\alpha}_k,\check{p}_k) =\tQ(\alpha_k,p_k)$, $\check{Q}(\check{\alpha},\check{p}) =\tQ(\alpha,p)$, and
\EEE
\eqref{1105210945} is equivalent to $\check{Q}(\check{\alpha},\check{p}) \leq \liminf_k \check{Q}(\check{\alpha}_k,\check{p}_k)$, so that we may assume to work with the restriction on $\Omega$ for $\tQ$.

Up to a subsequence, we may assume that
\begin{equation*}
\liminf_{k\to +\infty} \tQ(\alpha_k, p_k)=\lim_{k\to +\infty} \tQ(\alpha_k, p_k) < +\infty
\end{equation*}
and that the measures $\mu_k \in \mathcal{M}^+_b(\Omega)$ 
defined by 
\begin{equation*}
\mu_k(A):=\int_{A \sm\BBB \{\alpha_k=1\}\EEE} \C_1(\alpha_k) p_k \colon p_k \dx \quad\text{for }A\subset \Omega \text{ Borel, }
\end{equation*} 
\BBB with $p_k=0$ in $\{\alpha_k=1\}$, \EEE are such that
\begin{equation*}
\mu_k \wtos \mu \quad\text{in }\mathcal{M}^+_b(\Omega)\,.
\end{equation*}
By the Besicovitch derivation theorem and the Radon-Nikodym decomposition for $\mu$ (cf.\ \cite[Theorem~2.22]{AFP}), the result will follow from the \BBB three \EEE
 estimates
\begin{equation}\label{1105211028}
\frac{\mathrm{d} \mu}{\mathrm{d} \Ln}(x_0) \geq \C_1(\alpha(x_0))p(x_0) \colon p(x_0) \quad\text{for $\Ln$- a.e.\ } x_0 \in \Omega \sm \{\alpha=1\},
\end{equation}
\begin{equation}\label{1505211152}
|p^s|(\{\wt\alpha>0\})=0,
\end{equation}
\begin{equation}\label{1008211732}
p=0 \quad \Ln\text{-a.e.\ in }  \{\alpha=1\}.
\end{equation}
Let us then prove these estimates.
\medskip
\paragraph*{\textbf{Proof of \eqref{1105211028}.}}
\medskip
\paragraph*{\textit{Step~1: Choice of the blow up point $x_0$.}}
We choose $x_0$ in a subset of \BBB $\Omega\sm\{\alpha=1\}$ \EEE of full $\Ln$-measure, satisfying the following conditions:
\begin{itemize}
\item[(i)] in $x_0$ there exists the Radon-Nikodym derivative of $\mu$ with respect to $\Ln$
\begin{equation}\label{1105211313}
\frac{\mathrm{d}\mu}{\mathrm{d}\Ln}(x_0)=\lim_{\varrho\to 0^+} \frac{\mu(B_\varrho(x_0))}{\omega_n \varrho^n}\,.
\end{equation}
\item[(ii)] $x_0$ is a Lebesgue point for $\nabla \alpha$. \\
This gives that $\alpha^{x_0}_\varrho(y):= \alpha(x_0+\varrho y)$, $y \in B_1$ is such that
\begin{equation}\label{1105211113}
\alpha^{x_0}_\varrho \to \alpha(x_0)\BBB <1\EEE\quad\text{in }H^1(B_1)\,,
\end{equation} 
since $\nabla \alpha^{x_0}_\varrho(y)= \varrho \nabla \alpha(x_0+\varrho y)$, and then
\begin{equation*}
\int_{B_1} |\nabla \alpha^{x_0}_\varrho|^2 \, \mathrm{d}y= \varrho^2 \int_{B_1}|\nabla \alpha(x_0+\varrho y)|^2 \, \mathrm{d}y= \varrho^2 \frac{\int_{B_\varrho(x)} |\nabla \alpha|^2 \dx}{\varrho^n}
\end{equation*}
which tends to 0 as $\varrho\to 0^+$ since $\frac{\int_{B_\varrho(x)} |\nabla \alpha|^2 \dx}{\omega_n\varrho^n} \to |\nabla \alpha(x_0)|$.
\item[(iii)] $u$ is approximately differentiable in $x_0$. \\
Then
\begin{equation*}
u^{x_0}_\varrho(y):= \frac{u(x_0+\varrho y) - u(x_0)}{\varrho},\quad y\in B_1
\end{equation*}
is such that
\begin{equation}\label{1105211148}
u^{x_0}_\varrho \to \ol u_0\quad\text{in } L^1(B_1)\,, \quad \text{for } \ol u_0(y):= \nabla u (x_0) y\,,\quad y\in B_1\,.
\end{equation}
Being $u \in BD(\Omega)$, it holds that $u^{x_0}_\varrho \in BD(B_1)$ and
\begin{equation}\label{1105212109}
\EE u^{x_0}_\varrho \wtos \EE u(x_0) \quad\text{in } \Mbu
\end{equation}
\item[(iv)] $x_0$ is a Lebesgue point for $e \in \Lnn$.\\
Then $e^{x_0}_\varrho(y):=e(x_0+\varrho y)$, $y \in B_1$ is such that
\begin{equation}\label{1105212103}
e^{x_0}_\varrho \to e(x_0) \quad\text{in }\Lnnu\,.
\end{equation}
\item[(v)] in $x_0$ there exists the Radon-Nikodym derivative of $p$ (and then of $\EE u$, in view of (iv)) with respect to $\Ln$
\begin{equation*}
p(x_0)=\frac{\mathrm{d}p}{\mathrm{d}\Ln}(x_0)=\lim_{\varrho\to 0^+} \frac{p (B_\varrho(x_0))}{\omega_n \varrho^n} \in \Mnn
\end{equation*}
and 
\begin{equation}\label{1205210939}
\EE u(x_0)=e(x_0) + p(x_0)\,.
\end{equation}
\end{itemize}
In fact, conditions (i)--(v) are satisfied in a subset of \BBB $\Omega\sm \{\alpha=1\}$ \EEE of full $\Ln$-measure (see \cite{ACD97} for (iii)).
\medskip
\paragraph*{\textit{Step~2: Blow up argument: change of variables.}}
For $x_0$ fixed as in Step~1, we perform a blow up procedure. Let us fix a sequence $(\varrho_h)_h$ converging to 0 such that $\mu\big(B_{\varrho_h}(x_0)\big)=0$ for every $h$ (notice that this holds for all $\varrho$ except countable many). Then, by \eqref{1105211313} we have that 
\begin{equation}\label{0203201239}
\begin{split}
\omega_n\,\frac{\mathrm{d} \mu}{\mathrm{d} \Ln}(x_0)&= \lim_{h\to \infty}\lim_{k\to \infty}\frac{\mu_k(B_{\varrho_h}(x_0))}{\varrho_h^n} = \lim_{h\to \infty}\lim_{k\to \infty}\frac{1}{\varrho_h^n}\int_{B_{\varrho_h}(x_0) \BBB \sm\{\alpha_k=1\}  \EEE} \hspace{-2em}\C_1(\alpha_k) p_k \colon p_k \dx  \,.
\end{split}
\end{equation} 
Consider the rescaling function $\lambda^{x_0,\varrho}\colon B_\varrho(x_0) \to B_1$ defined by $\lambda^{x_0,\varrho}(x):= \frac{x-x_0}{\varrho}$. We define in correspondence the measures
\begin{equation}\label{1105212024}
\EE_\varrho^k:= \frac{1}{\varrho^n}\lambda^{x_0,\varrho}_{\#} \EE u_k,\qquad p_\varrho^k:= \frac{1}{\varrho^n}\lambda^{x_0,\varrho}_{\#} p_k, \qquad \tilde{e}_\varrho^k:= \frac{1}{\varrho^n}\lambda^{x_0,\varrho}_{\#} e_k,
\end{equation}
obtained from $\EE u_k$, $e_k$, $p_k$ through the push-forward of $\lambda^{x_0, \varrho}$, denoted by $\lambda^{x_0,\varrho}_{\#}$.
A straightforward calculation shows that 
\begin{equation*}
\EE_\varrho^k=\EE u_\varrho^k, \quad \tilde{e}_\varrho^k = e_\varrho^k \Ln
\end{equation*}
for $u_\varrho^k\in BD(B_1)$, $e_\varrho^k \in L^2(B_1;\Mnn)$ given by
\begin{equation*}
u_\varrho^k(y):=\frac{u_k(x_0+\varrho y)-u_k(x_0)}{\varrho},\qquad e_\varrho^k(y):=e_k(x_0+\varrho y)\,.
\end{equation*}
\CCC By the finiteness of $\tQ(\alpha_k, p_k)$, we deduce that $|(p_\varrho^k)^s|(\{\wt\alpha_\varrho^k >0\})=0$. \EEE
We observe also that
\begin{equation}\label{1205211012}
\nabla u_\varrho^k(y)= \nabla u_k (x_0+\varrho y)\quad\text{ for $\Ln$-a.e.\ }y\in B_1.
\end{equation}
In fact, this holds in the approximate differentiability points of $u$, by a change of variable in the very definition of approximate differential.
Moreover, $|p_\varrho^k|(B_1)=\frac{|p_k(B_\varrho(x))|}{\varrho^n}$ for every $\varrho>0$. In particular, due to \eqref{0203201239}, $|p_{\varrho_h}^k|(B_1)$ is bounded uniformly in $h$, $k$ along the sequence $(\varrho_h)_h$.
Notice also that for every $\varrho>0$ we have that
\begin{equation}\label{1205210921}
\EE u_\varrho^k=e_\varrho^k+p_\varrho^k,\qquad \EE u_\varrho^k \wtos \EE u^{x_0}_\varrho\quad\text{in }\Mbu, \qquad e_\varrho^k \wto e^{x_0}_\varrho\quad\text{in }\Lnnu,
\end{equation}
and that, by our assumption that $\alpha_k \wto \alpha$ in $H^1(\Omega)$, the functions $\alpha_\varrho^k\in H^1(B_1)$ defined by $\alpha_\varrho^k(y):=\alpha_k(x_0+\varrho y)$ are such that
\begin{equation*}
\alpha_\varrho^k \wto \alpha^{x_0}_\varrho \quad \text{in }H^1(B_1).
\end{equation*}
Collecting all the conditions above, we can use a diagonal argument to find a subsequence $(\varrho_{h_k})_k$ such that, for
\begin{equation*}
v_k:=u_{\varrho_{h_k}}^k,\quad q_k:=p_{\varrho_{h_k}}^k, \quad \eta_k:=e_{\varrho_{h_k}}^k, \quad \beta_k:=\alpha_{\varrho_{h_k}}^k,
\end{equation*}
the following holds:
\begin{equation}\label{1205210932}
\begin{split}
& \EE v_k= q_k + \eta_k,\\ 
v_k \wtos \ol u_0 \quad \text{in }BD(B_1), \quad  q_k \wtos p(x_0)  \quad\text{in }&\Mbu,\quad  \eta_k \wto e(x_0) \quad\text{in }\Lnnu, \\
& \beta_k \wto \alpha(x_0) \quad \text{in }H^1(B_1),\\
\omega_n\,\frac{\mathrm{d} \mu}{\mathrm{d} \Ln}(x_0)=\lim_{k\to \infty}\frac{1}{\varrho_{h_k}^n}\int_{B_{\varrho_{h_k}}(x_0)\BBB \sm\{\alpha_k=1\}  \EEE} & \hspace{-4em}\C_1(\alpha_k) p_k \colon p_k \dx = \lim_{k\to \infty}\int_{B_1\BBB \sm\{\beta_k=1\}  \EEE} \hspace{-3em}\C_1(\beta_k) q_k \colon q_k \,\mathrm{d}y.
\end{split}
\end{equation}
We notice that the last equation above follows from a change of variable, recalling \eqref{1205211012}.
Moreover, we have
\begin{equation}\label{1205211303} 
|(q_k)^s|(\{\wt\beta_k>0\})=0\quad \BBB \text{and}\quad q_k=0 \text{ in }\{\beta_k=1\}. \EEE
\end{equation}
\medskip
\paragraph*{\textit{Step~3: Blow up argument: semicontinuity.}}
Since $\mu$ is a nonnegative measure, \eqref{1105211028} is satisfied if $\alpha(x_0)=0$, by \eqref{C2'}. Let us then fix $x_0$, satisfying (i)--(v) in Step~1, such that $\alpha(x_0)>0$.

By \cite[Theorem~7]{Eva90}, up to passing to a subsequence we have that for every $\delta \in [1,2)$ and every $\varepsilon>0$ there exists a relatively open set $\tilde{A}_\varepsilon\subset B_1$ such that
\begin{equation}\label{1205211105}
\wt\beta_k \to \alpha(x_0) \quad\text{uniformly on } B_1 \sm \tilde{A}_\varepsilon,\qquad \mathrm{Cap}_\delta(\tilde{A}_\varepsilon, B_1) < \varepsilon.
\end{equation}
By \cite[Lemma~3.1]{Lah17} (notice that in subsets of $\Rn$ the two notions of capacity in \cite{Eva90} and \cite{Lah17} are equivalent, namely they are the same up to constants) there exists a constant $C>0$, depending only on $n$, and a set $A_\varepsilon \supset \tilde{A}_\varepsilon$ with (the notion of capacity below are all relative to $B_1$)
\begin{equation}\label{1205211138}
\mathrm{Cap}_1(A_\varepsilon) \leq C \,\mathrm{Cap}_1(\tilde{A}_\varepsilon) \quad\text{and}\quad \mathrm{Per}(A_\varepsilon, B_1) \leq C \,\mathrm{Cap}_1(\tilde{A}_\varepsilon).
\end{equation}
Collecting \eqref{1205211105} and \eqref{1205211138} we have that
\begin{equation}\label{1205211141}
\wt\beta_k \to \alpha(x_0) \quad\text{uniformly on } B_1 \sm A_\varepsilon,\qquad \mathrm{Per}(A_\varepsilon, B_1) < C\, \varepsilon.
\end{equation}
In particular, 
\begin{equation}\label{1205211309}
\wt\beta_k \geq \frac{\alpha(x_0)}{2}>0 \quad \text{in }B_1 \sm A_\varepsilon,
\end{equation} 
for $k$ large enough. By the last condition in \eqref{1205210932} and \eqref{eq:C} we get that $q_k$ is uniformly bounded in $L^2(B_1 \sm A_\varepsilon;\Mnn)$. 

For fixed $\varepsilon>0$ consider the function
\begin{equation*}
\check{v}_k:= v_k \, \chi_{B_1 \sm A_\varepsilon}.
\end{equation*}
It is immediate that $\check{v}_k$ are in $GSBD^2(\Omega)$ and that $\E(\check{v}_k)= \E(v_k) \, \chi_{B_1 \sm A_\varepsilon}$, $J_{\check{v}_k} \subset J_{v_k} \cup (\partial^* A_\varepsilon \cap B_1)$ (recall that $\E(v)$ denotes the approximate symmetric gradient of a function $v \in G(S)BD$).
By \eqref{1205211303} and \eqref{1205211309} we get
\begin{equation*}
J_{\check{v}_k} \subset \partial^* A_\varepsilon \cap B_1, \qquad \E(\check{v}_k)= \eta_k \, \chi_{B_1 \sm A_\varepsilon} + q_k \, \chi_{B_1 \sm A_\varepsilon}\in \Lnnu.
\end{equation*}  
so that, by \eqref{1205211141}, the functions $\check{v}_k$ are bounded in $GSBD^2(B_1)$. 
From the fact that $\check{v}_k \to \ol u_0 \, \chi_{B_1 \sm A_\varepsilon} $ in $L^1(B_1; \Rn)$, applying \cite[Theorem~11.3]{DM13} or \cite[Theorem~1.1]{CC21JEMS} we obtain that $\E(\check{v}_k) \wto \E(u)(x_0)$ in $L^2(B_1 \sm A_\varepsilon; \Mnn)$ and then (recall the second line in \eqref{1205210932})
\begin{equation}\label{1205211403}
q_k \wto p(x_0) \quad\text{in } L^2(B_1 \sm A_\varepsilon; \Mnn)\,.
\end{equation}
Being $\tQ$ \BBB lower semicontinuous and also convex in \EEE
the second variable, by Ioffe-Olech Lower Semicontinuity Theorem (see \cite[Theorem~2.3.1]{But89})
\begin{equation*}
\begin{split}
|B_1 \sm A_\varepsilon|\, & \C_1(\alpha(x_0)) p(x_0) \colon p(x_0)= \int_{B_1 \sm A_\varepsilon}\C_1(\alpha(x_0)) p(x_0) \colon p(x_0) \,\mathrm{d}y \\&\leq \lim_{k\to \infty}\int_{B_1\BBB \sm\{\beta_k=1\}  \EEE} \hspace{-3em} \C_1(\beta_k) q_k \colon q_k \,\mathrm{d}y=\omega_n \,\frac{\mathrm{d} \mu}{\mathrm{d} \Ln}(x_0)\,.
\end{split}
\end{equation*}
By \eqref{1205211141} it holds that $\lim_{\varepsilon\to 0}\Ln(A_\varepsilon)=0$. Therefore, the above inequality gives \eqref{1105211028} by the arbitrariness of $\varepsilon$.
\medskip
\paragraph*{\textbf{Proof of \eqref{1505211152}.}} 
In this step we use a slicing procedure. We recall
the basic notation: fixed $\xi \in \Sn$, we let
\begin{equation*}\label{eq: vxiy2}
\Pi^\xi:=\{y\in \Rn\colon y\cdot \xi=0\},\quad B^\xi_y:=\{t\in \R\colon y+t\xi \in B\} \ \ \ \text{ for any $y\in \Rn$ and $B\subset \Rn$},
\end{equation*}
and for every function $v\colon B\to  \Rn$ and $t\in B^\xi_y$, let
\begin{equation*}\label{eq: vxiy}
v^\xi_y(t):=v(y+t\xi),\qquad \widehat{v}^\xi_y(t):=v^\xi_y(t)\cdot \xi.
\end{equation*}
By Fubini Theorem it holds that for every $\xi \in \Sn$
\begin{equation}\label{1505211220}
(\alpha_k)\xy \to \alpha\xy,\quad (\hat{u}_k)\xy \to \hat{u}\xy\quad \text{in }L^1(\Omega\xy) \quad\text{for }\hn\text{-a.e.\ }y \in \Pi_\xi.
\end{equation}
Recalling \eqref{C2}, we have that
\begin{equation*}
\begin{split}
C &\geq \liminf_{k\to +\infty} \bigg\{\int_\Omega \Big\{\gamma(\alpha_k) |p_k|^2 + |e_k|^2 + |\nabla \alpha_k|^2 \Big\} \dx + |\EE u_k|(\Omega) \bigg\} \\
& \geq \liminf_{k\to +\infty} \bigg\{\int_\Omega \Big\{\gamma(\alpha_k) |p_k \xi \cdot \xi|^2 + |e_k \xi \cdot \xi|^2 + |\nabla \alpha_k \cdot \xi|^2 \Big\} \dx + |\EE u_k \xi \cdot \xi|(\Omega) \bigg\}\\
& = \liminf_{k\to +\infty}  \int_{\Pi_\xi} (\mathrm{I}_k)\xy \dH(y) \geq  \int_{\Pi_\xi} \liminf_{k\to +\infty}(\mathrm{I}_k)\xy \dH(y)
\end{split}
\end{equation*}
by Fatou's Lemma, where
\begin{equation}\label{1505211227}
(\mathrm{I}_k)\xy:=\int_{\Omega\xy} \Big\{\gamma((\alpha_k)\xy) |(p_k \xi \cdot \xi)\xy|^2 + |(e_k \xi \cdot \xi)\xy|^2 + |\nabla (\alpha_k)\xy|^2 \Big\} \,\mathrm{d}t + |\mathrm{D}(\hat{u}_k)\xy|(\Omega\xy)
\end{equation}
Therefore, we may fix $y$ in a set of full $\hn$-measure of $\Pi_\xi$ and find, in correspondence to $y$, a subsequence $k_j$ (possibly depending on $y$) such that 
\begin{equation*}
\lim_{j\to +\infty} (\mathrm{I}_{k_j})\xy=\liminf_{k\to +\infty}(\mathrm{I}_k)\xy,\quad (\alpha_{k_j})\xy \wto \alpha\xy \quad\text{in }H^1(\Omega\xy),\quad (\hat{u}_{k_j})\xy \wtos \hat{u}\xy \quad\text{in }BV(\Omega\xy).
\end{equation*}
Moreover
\begin{equation*}
(\wt\alpha_{k_j})\xy \wto \wt\alpha\xy \quad\text{uniformly in }\Omega\xy
\end{equation*} 
passing to the continuous representatives of the slices (which are the slices of $\wt \alpha_k$, $\wt \alpha$, the quasicontinuous representatives of $\alpha_k$, $\alpha$).

Let us fix an open set $I$ compactly contained in $\{\wt\alpha\xy >0\}$. By the uniform convergence stated above, we have that $I$ is compactly contained in $\{(\wt\alpha_{k_j})\xy >0\}$, provided $j$ is large enough.
Being $(I_{k_j})\xy$ uniformly bounded in $j$ (and recalling \eqref{propgamma}), this implies that $p_{k_j}\xi \cdot \xi$ are equibounded in $L^2(I)$ with respect to $j$. Then $(\hat{u}_{k_j})\xy$ is equibounded in $H^1(I)$ with respect to $j$ and, by \eqref{1505211220},
\begin{equation*}
(\hat{u}_{k_j})\xy \wto \hat{u}\xy \quad\text{in }H^1(I).
\end{equation*}
Therefore $|\mathrm{D}^s \hat{u}\xy|(I)= |(p^s \xi \cdot \xi)\xy|(I)=0$. By the arbitrariness of $I$, we have $|\mathrm{D}^s \hat{u}\xy|(\{\wt\alpha\xy >0\})= |(p^s \xi \cdot \xi)\xy|(\{\wt\alpha\xy >0\})=0$.

We notice that, setting $B:=\{\wt\alpha>0\}$, the sets $\{\wt\alpha\xy >0\}$ are the slices $B\xy$ of $B$ for $\hn$-a.e.\ $y\in \Pi_\xi$. By the structure theorem for $BD$ functions proven in \cite[Theorem~4.5]{ACD97} we deduce that
\begin{equation*}
|\EE^s u \xi \cdot \xi|(\{\wt\alpha>0\})=|p^s \xi \cdot \xi|(\{\wt\alpha>0\})=0.
\end{equation*}
By the arbitrariness of $\xi \in \Sn$ we conclude \eqref{1505211152}.
\medskip
\paragraph*{\textbf{Proof of \eqref{1008211732}.}} We argue exactly as in the proof of \eqref{1105211028},
 with the only difference that now $\alpha(x_0)=1$. At this stage, for every fixed $\delta>0$ 
\begin{equation}\label{1108211233}
\beta_k \geq \alpha(x_0)-\delta=1-\delta \quad\text{in }B_1\sm A_\varepsilon
\end{equation}
for $k$ large enough.
We have also in this case that 
\begin{equation*}
q_k \wto p(x_0) \quad\text{in } L^2(B_1 \sm A_\varepsilon; \Mnn)\,,
\end{equation*}
and since $\int_{B_1 \sm\{\beta_k=1\}} \C_1(\beta_k) q_k \colon q_k \,\mathrm{d}y$ are equibounded (with $q_k=0$ in $\{\beta_k=1\}$) and $p(x_0)\in \Mnn$ is constant, by \eqref{1108211233} and \eqref{C5} we conclude that $p(x_0)=0$.
\end{proof}
\begin{rem}\label{rem:1708211308}
The proof goes exactly in the same way if a $L^\gamma$ gradient damage term is present, $\gamma>1$, in place of a $L^2$ term. In fact, the weak $W^{1,\gamma}$ convergence implies \eqref{1205211105} as well.
\end{rem}
\begin{rem}\label{rem:1108211011}
Consider the case where \eqref{C1} and \eqref{C5} are replaced by 
\begin{equation}\label{C1'}
 \C_1 \in C([0,1]; Lin(\Mnn;\Mnn)).
\end{equation}
In this case Theorem~\ref{thm:1505212332} still holds. The proof goes as above, without imposing that plastic variables are null where damage variables are 1, and without proving \eqref{1008211732}.
If $\C_1(0) \neq 0$ and $\gamma(0)>0$ in \eqref{eq:C}, then Theorem~\ref{thm:1505212332} follows directly by applying Ioffe-Olech Lower Semicontinuity Theorem (see \cite[Theorem~2.3.1]{But89}) for every $\alpha_k\wto \alpha$ in $H^1(\Omega;[0,1])$ and every $p_k\wto p$ in $\MbDd$.
\end{rem}
  
 \begin{prop}\label{prop:1505212300} 
 The minimisation problem \eqref{1005211029} admits solutions.
 \end{prop} 
 \begin{proof}
 By the integration by parts formula \eqref{1605211221}
 we get
 that \eqref{1005211029} is equivalent to
 \begin{equation*}
 \begin{split}
 \min_{0\leq \alpha\leq \alpha_k^{i-1},\, (u,e,p) \in \A(\wki)} \Big\{ &\QQ(e) -  \langle \varrho_k^i, e \rangle +D(\alpha)+\da + \tQ(\alpha,p) 
 \\&+ \HH(p-p_k^{i-1}) - \int_{\Omega\cup \dod} \varrho_k^i \colon \mathrm{d}( p -p_k^{i-1})
 \Big\},
 \end{split}
 \end{equation*}
 since the functionals to minimise differ by constant terms.
 
 By \eqref{0302211656} and Cauchy Inequality
 \begin{equation*}
 \QQ(e) -  \langle \varrho_k^i, e \rangle \geq \frac{\gamma_1}{2} \| e\|_2^2  - \frac{1}{2 \gamma_1} \|\varrho_k^i\|_2^2.
 \end{equation*}
 With \eqref{1505212325} and the constitutive assumptions in Section~\ref{Sec:2}, we deduce that every minimising sequence $(\beta_k, v_k, \eta_k, q_k)$ is bounded in $H^1(\Omega) \times BD(\Omega) \times L^2(\Omega) \times \MbDd$. 
Arguing as in the last part of \cite[Theorem~3.3]{Mora} we see the lower semicontinuity of 
$\HH(p-p_k^{i-1}) -\int_{\Omega\cup \dod} \varrho_k^i\colon \mathrm{d} (p -p_k^{i-1})$ 
with respect to the weak$^*$ convergence of $p \in \MbDd$. The existence of solutions follows from the Direct Method of Calculus of Variations: it is enough to apply Theorem~\ref{thm:1505212332} and the fact that the remaining terms are directly weakly lower semicontinuous in the target spaces of the respective variables.
 \end{proof}
 
From the solutions to the incremental minimisation problems, we define their interpolations in time by setting for every $t\in \ot$
\begin{equation} \label{interp}
\begin{split}
\alpha_k(t):= \aki,\, \,\, \uk(t):=\uki,\, \,\, e_k(t):=e_k^i, \,\,\,    p_k(t):=p_k^i 
\\
\sigma_k(t):= \C(\aki) e_k^i, \,\,\, \LL_k(t):=\LL_k^i, \,\,\,w_k(t):= \wki\,.
\end{split}
\end{equation}
where $i$ is the largest integer such that $\tki\le t$, that is $i$ is the integer part of $\frac{t}{T}k$. 

By definition $t \mapsto \alpha_k(t)$ is nonincreasing, $( u_k(t),e_k(t),p_k(t))\in \A(w_k(t))$ for every $t\in\ot$, and, by \eqref{1005211029} together with the triangle inequality for $\HH$, we obtain the discrete time stability condition
\begin{equation}\label{1905210948}\tag{ST$_k$}
\begin{split}
\QQ(e_k(t))& +D(\alpha_k(t))+\| \nabla \alpha_k(t)\|_2^2 + \tQ(\alpha_k(t),p_k(t))  - \langle \LL_k(t), u_k(t) \rangle \\& \leq  \QQ(\eta)+D(\beta)+\| \nabla \beta\|_2^2 + \tQ(\beta, q)  +\HH(q-p_k(t)) - \langle \LL_k(t), v \rangle
\end{split}
\end{equation}
for every $\beta \leq \alpha_k(t)$, $\beta\geq 0$, $(v,\eta,q) \in \A(w_k(t))$.

Moreover, testing the minimisation problem \eqref{1005211029} with $\alpha_k^{i-1}$, $\big(u_k^{i-1}+(w_k^i-w_k^{i-1}), e_k^{i-1} + \EE(w_k^i-w_k^{i-1}),p_k^{i-1}\big) \in \A(w_k^i)$ we get
\begin{equation}\label{1905211046}
\begin{split}
\QQ(e_k^i)&+D(\alpha_k^i)+ \| \nabla \alpha_k^i\|_2^2 + \tQ(\alpha_k^i,p_k^i) + \HH(p_k^i-p_k^{i-1}) - \langle \LL_k^i, u_k^i \rangle 
\\&\leq  \QQ(e_k^{i-1}) + \langle \sigma_k^i, \EE(w_k^i-w_k^{i-1})\rangle +\gamma_2\| \EE(w_k^i-w_k^{i-1})\|_2^2 + D(\alpha_k^{i-1})+ \| \nabla \alpha_k^{i-1}\|_2^2 
\\& \hspace{1em} + \tQ(\alpha_k^{i-1},p_k^{i-1}) - \langle \LL_k^i, u_k^{i-1} + (w_k^i-w_k^{i-1}) \rangle.
\end{split}
\end{equation}
We manipulate the last term in the inequality above as
\begin{equation}\label{1905211037}
\begin{split}
\langle \LL_k^i, u_k^{i-1} + (w_k^i-w_k^{i-1}) \rangle= &\langle \LL_k^{i-1}, u_k^{i-1} \rangle + \langle \LL_k^i, w_k^i \rangle -\langle \LL_k^{i-1}, w_k^{i-1} \rangle\\&
\hspace{1em} + \int_{t_k^{i-1}}^{t_k^i} \langle \dot{\LL}(s), u_k(s) - w_k(s) \rangle \,\mathrm{d}s,
\end{split}
\end{equation}
since, by our assumptions on $\varrho$, $\dot{\LL}(t)\in BD(\Omega)'$ for a.e.\ $t\in \ot$ and every $t \mapsto \langle \dot{\LL}(t), v(t) \rangle$ is in $L^1(0,T)$ for $v \in L^\infty(0,T;BD(\Omega))$ (cf.\ \cite[Remark~4.1]{Mora}).
Combining \eqref{1905211037} and \eqref{1905211046}, and summing on $i \in \{0,\dots, k\}$, 
we obtain the discrete time energy inequality
\begin{equation}\label{1905211048}\tag{EI$_k$}
\begin{split}
\QQ(e_k(t))& +D(\alpha_k(t))+\| \nabla \alpha_k(t)\|_2^2 + \tQ(\alpha_k(t),p_k(t)) + \V_\HH(p_k;0,t) - \langle \LL_k(t), u_k(t) -w_k(t) \rangle \\
& \leq \QQ(e_0) +D(\alpha_0)+\| \nabla \alpha_0\|_2^2 + \tQ(\alpha_0,p_0) - \langle \LL_0, u_0 -w_0 \rangle +\int_0^{t_k^i} \langle \sigma_k(s), \EE \dot{w}(s) \rangle \,\mathrm{d}s \\& \hspace{1em} - \int_0^{t_k^i} \langle \dot{\LL}(s), u_k(s) - w_k(s) \rangle \,\mathrm{d}s + \delta_k,
\end{split}
\end{equation}
where 
\begin{equation}\label{2305211506}
\delta_k:=\gamma_2 \bigg(\max_{1\leq r \leq k} \int_{t_k^{r-1}}^{t_k^r} \|\EE \dot{w}(s)\|_2 \,\mathrm{d}s \bigg)  \int_{0}^{t_k^i} \|\EE \dot{w}(s)\|_2 \,\mathrm{d}s  \to 0\quad\text{as }k\to +\infty.
\end{equation}
Notice that in \eqref{1905210948} we 
exploit the fact that $\V_\HH(p_k;0,t)=\sum_{j=1}^i \HH(p_k^j-p_k^{j-1})$, $p_k$ being piecewise constant in time.
By the integration by parts formula in \eqref{1905211046}, we have 
\begin{equation*}
\begin{split}
\langle \LL_k^i, u_k^i -u_k^{i-1} - (w_k^i-w_k^{i-1}) \rangle &= \int_{\Omega\cup\dod}\varrho_k^i \colon \mathrm{d}(p_k^i-p_k^{i-1})  + \langle \varrho_k^i, e_k^i - \EE w_k^i \rangle \\&  \hspace{1em} - \langle \varrho_k^{i-1}, e_k^{i-1}- \EE w_k^{i-1} \rangle + \int_{t_k^{i-1}}^{t_k^i} \langle \dot{\varrho}(s), e_k(s) - \EE w_k(s) \rangle \,\mathrm{d}s,
\end{split}
\end{equation*}
so that, summing \eqref{1905211046} over $i$, we may recast \eqref{1905211048} in
 \begin{equation}\label{1905211737}
\begin{split}
\QQ(e_k(t))& - \langle \varrho_k(t), e_k(t)-\EE w_k(t) \rangle +D(\alpha_k(t))+\| \nabla \alpha_k(t)\|_2^2 + \tQ(\alpha_k(t),p_k(t)) \\&
\hspace{1em}+ \sum_{1\leq i \leq k} \big\{ \HH(p_k^i-p_k^{i-1})-\int_{\Omega\cup\dod}\varrho_k^i \colon \mathrm{d}(p_k^i-p_k^{i-1})  \big\} - \langle \LL_k(t), u_k(t) \rangle \\
& \leq \QQ(e_0) - \langle \varrho_0, e_0-\EE w_0 \rangle +D(\alpha_0)+\| \nabla \alpha_0\|_2^2 + \tQ(\alpha_0,p_0) - \langle \LL_0, u_0 \rangle  \\& \hspace{1em} +\int_0^{t_k^i} \langle \sigma_k(s), \EE \dot{w}(s) \rangle \,\mathrm{d}s - \int_0^{t_k^i} \langle \dot{\varrho}(s), e_k(s) - \EE w_k(s) \rangle \,\mathrm{d}s + \delta_k.
\end{split}
\end{equation}
Expressing the discrete time energy inequality as above, it is now readily seen that, arguing as done in Proposition~\ref{prop:1505212300} to prove coercivity (and using the integrability assumptions on $w$, $\varrho$), the a priori bounds
\begin{equation}\label{1905211744}
\sup_{t\in \ot} \Big(\|e_k(t)\|_2 + \|\alpha_k(t)\|_{H^1}\Big)\leq C,\qquad \V(p_k; 0,T) \leq C
\end{equation}
hold, for $C>0$ independent of $k$.
 \section{Passage to the continuous time limit}

This section contains the proof of existence of quasistatic evolutions, namely of Theorem~\ref{thm:2305211857}, and of some properties of these evolutions.
We divide the exposition into subsections, to ease the reading.

\subsection{Compactness}

In view of the a priori bound in \eqref{1905211744} on the variations in time of $p_k$, by a generalized version of the Helly theorem (cf.\ \cite[Lemma~7.2]{Mora}), there are a (not relabeled) subsequence and a function $p\colon \ot \to \MbDd$ with bounded variation on $\ot$ such that
\begin{equation*}
p_k(t) \wtos p(t) \quad\text{in }\MbDd \quad\text{for every }t\in \ot.
\end{equation*}
Moreover, since the functions $\alpha_k\colon \ot \to H^1(\Omega)$ are nonincreasing in time and $\alpha_k(t)$ are all valued in $[0,1]$, by a Helly-type theorem (cf.\ \cite{MalDuc} and \cite[Subsection~4.4]{Cri16}) there exist a (not relabeled) subsequence and a function $\alpha \colon \ot \to H^1(\Omega)$ nonincreasing in time such that
\begin{equation*}
\alpha_k(t) \wto \alpha(t) \quad\text{in }H^1(\Omega) \quad\text{for every }t\in \ot.
\end{equation*}
By \eqref{1905211744}, for every fixed $t\in \ot$ there exists a subsequence $k_j$, possibly depending on $t$, such that
\begin{equation}\label{1208210858}
e_{k_j}(t) \wto \ol e \quad\text{in }\Lnn,\qquad u_{k_j}(t) \wtos \ol u \quad\text{in }BD(\Omega).
\end{equation}

\subsection{Stability condition}

By the following theorem, we deduce that the discrete time stability condition \eqref{1905210948} passes to the limit under the available convergences.
\begin{thm}\label{thm:1905212052}
Let $w_k \in H^1(\Omega;\Rn)$
$\LL_k \in BD(\Omega)'$, 
$\alpha_k\in H^1(\Omega;[0,1])$, 
$(u_k, e_k, p_k)\in \A(w_k)$ such that $w_k \to w_\infty$ in $H^1(\Omega;\Rn)$, $\LL_k \to \LL_\infty$ in $BD(\Omega)'$, $\alpha_k \wto \alpha_\infty$ in $H^1(\Omega)$, $u_k\wtos u_\infty$ in $BD(\Omega)$, $e_k\wto e_\infty$ in $\Lnn$, $p_k\wtos p_\infty \in \MbDd$. Then $(u_\infty, e_\infty, p_\infty)\in \A(w_\infty)$ and, if
\begin{equation}\label{1905212057}
\begin{split}
\QQ(e_k)& +D(\alpha_k)+\| \nabla \alpha_k\|_2^2 + \tQ(\alpha_k,p_k)  - \langle \LL_k, u_k \rangle \\& \leq  \QQ(\eta)+D(\beta)+\| \nabla \beta\|_2^2 + \tQ(\beta, q)  +\HH(q-p_k) - \langle \LL_k, v \rangle
\end{split}
\end{equation}
for every $k$ and $\beta \leq \alpha_k$, $\beta\geq 0$, $(v,\eta,q) \in \A(w_k)$, then
\begin{equation}\label{1905212058}
\begin{split}
\QQ(e_\infty)& +D(\alpha_\infty)+\| \nabla \alpha_\infty\|_2^2 + \tQ(\alpha_\infty,p_\infty)  - \langle \LL_\infty, u_\infty \rangle \\& \leq  \QQ(\eta)+D(\beta)+\| \nabla \beta\|_2^2 + \tQ(\beta, q)  +\HH(q-p_\infty) - \langle \LL_\infty, v \rangle
\end{split}
\end{equation}
for every $\beta \leq \alpha_\infty$, $\beta\geq 0$, $(v,\eta,q) \in \A(w_\infty)$.
\end{thm}

\begin{proof}
By the very same argument of \cite[Lemma~2.1]{Mora}, it holds that $(u_\infty,e_\infty,p_\infty)\in \A(w_\infty)$.

First we fix the test functions in the limit stability problem: $0\leq\beta \leq \alpha_\infty$, and $(v, \eta, q)\in \A(w_\infty)$. Then we consider the following test functions for \eqref{1905212057}:
\begin{align*}
 \beta_k := \beta \wedge \alpha_k,\quad v_k :=v-u_\infty +u_k, \quad \eta_k := \eta-e_\infty +e_k, \quad  q_k := q-p_\infty+p_k.
\end{align*}
It holds that $\beta_k \wto \beta$ and $ \beta \vee \alpha_k \wto \alpha_\infty$ 
in $H^1(\Omega)$,  $v_k\wtos v$ 
in $BD(\Omega)$, $\eta_k\wto \eta$ 
in $\Lnn$, $q_k\wtos q$ 
in $\MbDd$.

We remark that we may assume that all the terms in $\tQ$ in \eqref{1905212058} and that one in the left handside of \eqref{1905212057} are finite: in fact, the left handside in \eqref{1905212057} is finite since the minimum problem has a solution, see Proposition~\ref{prop:1505212300}; the left handside in \eqref{1905212058} is finite by Theorem~\ref{thm:1505212332}; if the right handside in \eqref{1905212058} is infinity there is nothing to prove. Therefore
 \begin{equation}\label{2105211733}
 \begin{split}
 q \in L^2(\{\beta>0\};\Mnn)& \subset L^2(\{\beta_k>0\};\Mnn), \quad q=0\text{ in } \{\beta=1\}\supset \{\beta_k=1\}, \\
  p_\infty  \in L^2(\{\alpha_\infty>0\};\Mnn) &\subset L^2(\{\beta_k>0\};\Mnn),\quad p_\infty=0\text{ in } \{\alpha_\infty=1\}\supset \{\beta_k=1\}, \\
   p_k  \in L^2(\{\alpha_k>0\};\Mnn) &\subset L^2(\{\beta_k>0\};\Mnn), \quad p_k=0\text{ in } \{\alpha_k=1\}\supset \{\beta_k=1\},
  \end{split}
 \end{equation}
 since $\beta_k\leq \beta\leq \alpha_\infty \BBB \leq 1$ and $\beta_k\leq \alpha_k\leq 1$. It follows that also the right handside in \eqref{1905212057} is finite. To ease the reading we omit the indication that the plastic strain is null where the damage variable is 0, when writing the integral defining $\tQ$.
 
Inserting the test functions in  \eqref{1905212057}, we thus obtain the equivalent inequality
\begin{equation}\label{2005210638}
\begin{split}
 D(\alpha_k)&+\| \nabla (\beta \vee \alpha_k)\|_2^2 - \| \nabla \beta\|_2^2 + \tQ(\alpha_k,p_k) - \tQ(\beta_k,p_k)  - \langle \LL_k, u_\infty \rangle \\& \leq \tfrac12 \langle \C (\eta-e_\infty + 2 e_k), \eta -e_\infty \rangle+D(\beta_k) + \langle \C_1(\beta_k) (q-p_\infty +2 p_k) , q-p_\infty \rangle \\& \hspace{1em}+\HH(q-p_\infty) - \langle \LL_k, v \rangle
\end{split}
\end{equation}
by subtracting $\tQ(\beta_k,p_k)$ from both sides, by using the modularity condition 
\[
\| \nabla (\alpha_1 \vee \alpha_2) \|_{2}^{2} +\| \nabla (\alpha_1 \wedge \alpha_2) \|_{2}^{2}     = \| \nabla \alpha_1 \|_{2}^{2} + \| \nabla \alpha_2 \|_{2}^{2}\,,
\]
 for every $\alpha, \beta \in \WW$, and from the fact that
 \begin{equation}\label{diffquad}
 \tQ(\alpha, p_1)-\tQ(\alpha, p_2)=\langle \C_1(\alpha) (p_1+p_2) , p_1-p_2\rangle
 \end{equation}
  for every $\alpha \in \WW$ and $p_1, p_2 \in L^2(\{\alpha>0\};\Mnn)$. 

   In view of the convergences in the hypotheses plus those obtained above for the test functions,
   we have that $\| \nabla \alpha_k\|_2^2\leq \liminf_k \| \nabla (\beta \vee \alpha_k)\|_2^2$, that $D(\alpha_k)$, $D(\beta_k)$, $\langle \LL_k, v \rangle$, $\langle \LL_k, v_\infty \rangle$ converge to $D(\alpha_\infty)$, $D(\beta)$, $\langle \LL_\infty, v \rangle$, $\langle \LL_\infty, v_\infty \rangle$, and that
  \begin{equation}\label{2005210711}
  \lim_{k\to +\infty} \tfrac12\langle \C (\eta-e_\infty + 2 e_k), \eta -e_\infty \rangle = \tfrac12\langle \C (\eta+ e_\infty), \eta -e_\infty \rangle= \QQ(\eta) -\QQ(e_\infty).
  \end{equation}
Thus it lasts to investigate the terms involving $\tQ$. First, it holds that
\begin{equation}\label{2005210721}
\tQ(\alpha_\infty,p_\infty) - \tQ(\beta,p_\infty)\leq\liminf_{k\to +\infty} \Big\{\tQ(\alpha_k,p_k) - \tQ(\beta_k,p_k) \Big\}.
\end{equation}
In fact, the functional $\widehat \QQ(\alpha_1, \alpha_2,q)= \tQ(\alpha_1,q) - \tQ(\alpha_2,q) $ given by
\begin{equation}
\widehat \QQ(\alpha_1, \alpha_2,q):=  
\begin{dcases}
\int_\Omega \big[\C_1(\alpha_1)-\C_1(\alpha_2)\big] q \colon q \dx &\ \text{if }|q^s|(\{\wt\alpha_1>0\})=0,\, q\,\chi_{\{\alpha_1=1\}}=0,\, \alpha_1\geq \alpha_2\\
+\infty &\ \text{otherwise},
\end{dcases}
\end{equation}
assumes nonnegative values by \eqref{C3}, is semicontinuous with respect to the strong $L^1$ convergence of $(\alpha_1, \alpha_2)$, and is convex and lower semicontinuous in $q$. Therefore, by the very same arguments of Theorem~\ref{thm:1505212332} we get that $\widehat \QQ(\alpha_\infty, \beta,p_\infty)\leq \liminf_k \widehat \QQ(\alpha_k, \beta_k,p_k)$, namely \eqref{2005210721}.

Second, we claim that 
\begin{equation}\label{2005210758}
\lim_{k\to +\infty} \langle \C_1(\beta_k) (q-p_\infty +2 p_k) , q-p_\infty \rangle= \langle \C_1(\beta) (q+p_\infty) , q-p_\infty \rangle = \tQ(\beta, q) - \tQ(\beta, p_\infty).
\end{equation}
(We still omit the indication that the plastic strain is null where the damage variable is 0, when writing the integral defining $\tQ$.)
Since $\beta_k \wto \beta$ in $H^1(\Omega)$, arguing as done for \eqref{1205211141} with \cite[Lemma~3.1]{Lah17}, for every $\varepsilon>0$ there is an open set $D_\varepsilon \subset \Omega$ with $\mathrm{Per}(D_\varepsilon,\Omega)+ \Cap_1(D_\varepsilon,\Omega)<C\,\varepsilon$ such that 
\begin{equation}\label{2005212331}
\beta_k \to \beta \quad\text{ uniformly in }\Omega\sm D_\varepsilon.
\end{equation}
We notice that, by \eqref{C2} (cf.\ also the derivation of \eqref{C4} from \eqref{0302211656}), the functions $\C_1(\beta_k) p_k$
are equibounded with respect to $k$ in $\Lnn$. 
In view of \eqref{2105211733}, 
we deduce that the functions $\C_1(\beta_k) (q-p_\infty +2p_k)$ are in  $\Lnn$ and thus
\begin{equation}\label{2005212310}
\C_1(\beta_k) (q-p_\infty +2p_k) \colon (q-p_\infty) \quad\text{are equiintegrable with respect to }k.
\end{equation} 
Since $\beta_k \leq \beta$, recalling \eqref{C2'} it holds that 
\begin{equation}\label{2105211737}
\C_1(\beta_k) (q-p_\infty +2p_k) \colon (q-p_\infty)= \C_1(\beta_k) (q-p_\infty +2p_k) \colon (q-p_\infty) \chi_{\{\beta>0\}}.
\end{equation}
We observe also that $E_\delta:=\{\beta>\delta\}$ are increasing as $\delta$ decreases and $\chi_{E_\delta}\to \chi_{\{\beta>0\}}$ in $L^1(\Omega)$ as $\delta\to 0$. Then
\begin{equation}\label{2005212311}
\lim_{\delta\to 0} \Ln\big(\{\beta>0\} \sm E_\delta \big)=0.
\end{equation}
By \eqref{2005212310}, \eqref{2105211737}, \eqref{2005212311}, and the properties of $D_\varepsilon$, we get
\begin{equation}\label{2005212326}
\lim_{\varepsilon,\delta \to 0} \Big| \langle \C_1(\beta_k) (q-p_\infty +2 p_k) , q-p_\infty \rangle - \int_{E_{\varepsilon,\delta}}  \C_1(\beta_k) (q-p_\infty +2 p_k) \colon (q-p_\infty) \dx  \Big|=0 
\end{equation}
for
\begin{equation*}
E_{\varepsilon,\delta}:=E_\delta\sm D_\varepsilon=\{\beta>\delta\}\sm D_\varepsilon
\end{equation*}
uniformly in $k$.

By \eqref{2005212331} it follows that $\{\beta>\delta\}\sm D_\varepsilon \subset \{\beta_k>\frac{\delta}{2}\}$ for $k$ large enough. 
Then $p_k \chi_{E_{\varepsilon,\delta}}$ are equibounded in $k$ in $\Lnn$. Furthermore, being $\beta\in H^1(\Omega)$, by the Coarea Formula it is not restrictive to assume that $\{\beta> \delta\}$ has finite perimeter (this holds for a.e.\ $\delta$, and we are just interested to arbitrarily small $\delta$) so that $u_k \chi_{E_{\varepsilon,\delta}}$ are equibounded in $GSBD^2(\Omega)$. By assumption, we get
\begin{equation*}
u_k \chi_{E_{\varepsilon,\delta}} \to u_\infty \chi_{E_{\varepsilon,\delta}} \in GSBD^2(\Omega) \quad\Ln\text{-a.e.\ in }\Omega,
\end{equation*}
and this, with $e_k \wto e_\infty$ in $\Lnn$, gives (recall e.g.\ \cite[Theorem~11.3]{DM13}) 
\begin{equation}\label{2005212343}
p_k \chi_{E_{\varepsilon,\delta}} \wto p_\infty \chi_{E_{\varepsilon,\delta}} \quad\text{in } \Lnn.
\end{equation}
By \eqref{2005212331} and \eqref{2005212343}
\begin{equation}\label{2005212347}
\begin{split}
\lim_{k\to +\infty}\int_{E_{\varepsilon,\delta}} \C_1(\beta_k) (q-p_\infty +2 p_k) \colon (q-p_\infty) \dx = \int_{E_{\varepsilon,\delta}}  \C_1(\beta) (q+p_\infty) \colon (q-p_\infty) \dx,
\end{split}
\end{equation}
using also that the integrals above are evaluated outside $\{\beta=1\}$, since $q=0$, $p_\infty=0$ therein, and that $\C_1(\beta_k)$ converge uniformly to $\C_1(\beta)$ on $\{\beta=1\}$ (by \eqref{C1}, \eqref{2005212331}, and since $\beta_k\leq \beta<1$).

We now deduce \eqref{2005210758} by collecting \eqref{2005212326}, \eqref{2005212347}, and the facts that $\lim_{\varepsilon,\delta\to 0} \Ln(\{\beta>0\} \sm E_{\varepsilon,\delta})=0$ and $\C_1(\beta_\infty) (q+p_\infty) \colon (q-p_\infty) \in L^1(\Omega)$.

All in all, by \eqref{2005210711}, \eqref{2005210721}, and \eqref{2005210758} we conclude the proof.
\end{proof}

\begin{rem}\label{rem:1108211021}
Assuming \eqref{C1'} in place of \eqref{C1}, \eqref{C5}, Theorem~\ref{thm:1905212052} holds with the same proof (with obvious simplifications, e.g. there is no need for the second conditions in \eqref{2105211733}). If $\C_1(0)$ is positive definite, we can treat $\tQ$ as done for $Q$ in \eqref{2005210711}, to prove directly \eqref{2005210758}.
\end{rem}

Theorem~\ref{thm:1905212052} allows us to pass to the limit \eqref{1905210948} along the subsequence $k_j$ satisfying \eqref{1208210858}, to obtain for every $t\in\ot$
\begin{equation*}
\begin{split}
\QQ(\ol e)& +D(\alpha(t))+\| \nabla \alpha(t)\|_2^2 + \tQ(\alpha(t),p(t))  - \langle \LL(t), \ol u \rangle \\& \leq  \QQ(\eta)+D(\beta)+\| \nabla \beta\|_2^2 + \tQ(\beta, q)  +\HH(q-p(t)) - \langle \LL(t), v \rangle
\end{split}
\end{equation*}
for every $0\leq\beta \leq \alpha(t)$, $(v,\eta,q) \in \A(w(t))$.
In particular, taking $\beta=\alpha(t)$, it holds that $(\ol u, \ol e)$ minimises
\begin{equation*}
F(t)\colon (v,\eta)\mapsto \QQ(\eta)-\langle \LL(t), v \rangle
\end{equation*}
on the convex set $\wt K(t):=\{(v,\eta)\colon (v,\eta,p(t))\in\A(w(t))\}$. This implies that $(v,\eta)$ is uniquely determined: in fact, if $(v_1,\eta_1)$, $(v_2,\eta_2)$ are different minimisers, then both $v_1 \neq v_2$ and $\eta_1 \neq \eta_2$, by \eqref{1605210915}; by the strict convexity of $\QQ$ and the linearity of $\langle \LL(t), \cdot \rangle$ we would have $F\big(\frac{(v_1,\eta_1)+(v_2,\eta_2)}{2}\big)< \frac12\big(F(v_1,\eta_1) + F(v_2,\eta_2)\big)$, which is a contradiction.

Therefore, setting $u(t):=\ol u$, $e(t):=\ol e$, we obtain that
\begin{equation*}
u_k(t)\wtos u(t)\quad \text{in }BD(\Omega),\qquad e_k(t) \wto e(t) \quad\text{in }\Lnn
\end{equation*}
for the sequence independent of $t\in\ot$. Moreover, for every $t\in\ot$ it holds the stability condition:
\begin{equation}\label{2205210000}\tag{ST}
\begin{split}
\QQ(e(t))& +D(\alpha(t))+\| \nabla \alpha(t)\|_2^2 + \tQ(\alpha(t),p(t))  - \langle \LL(t), u(t) \rangle \\& \leq  \QQ(\eta)+D(\beta)+\| \nabla \beta\|_2^2 + \tQ(\beta, q)  +\HH(q-p(t)) - \langle \LL_k(t), v \rangle
\end{split}
\end{equation}
for every $\beta \leq \alpha(t)$, $(v,\eta,q)\in \A(w(t))$.

\subsection{Weak continuity a.e.\ in time}

We notice that Theorem~\ref{thm:1905212052} allows us also to infer that
\begin{equation}\label{2305211732}
e(s) \wto e(t) \quad\text{in }\Lnn,\ u(s) \wtos u(t) \quad\text{in }BD(\Omega)\quad\text{for }s\to t, \text{ for a.e.\ }t\in \ot.
\end{equation}
In fact, first we have that
\begin{equation}\label{2305211828}
\alpha(s) \wto \alpha(t) \quad\text{in }H^1(\Omega),\ p(s) \wtos p(t) \quad\text{in }\MbDd\quad\text{for }s\to t, \text{ for a.e.\ }t\in \ot,
\end{equation}
which hold indeed for any $t$ except at most countable many: 
the weak continuity of $\alpha$ follows by \cite[Lemma~A.2]{Cri16} plus the uniform boundedness of $\alpha(t)$ in $H^1(\Omega)$, the weak$^*$ continuity of $p$ follows from the fact that
$p$ has bounded variation with values in $\MbDd$. Then, if $t \in \ot$ is a point of weak continuity for both $\alpha$ and $p$, using the strong continuity in time of the loading and Theorem~\ref{thm:1905212052} we get that the weak limits of $e(s)$, $u(s)$ minimises $F(t)$ on $\wt K(t)$, and so they coincide with $e(t)$, $u(t)$.

\subsection{Energy balance}

The discrete time energy inequality \eqref{1905211048} passes to the limit into
\begin{equation}\label{2205210006}\tag{EI}
\begin{split}
\QQ(e(t))& +D(\alpha(t))+\| \nabla \alpha(t)\|_2^2 + \tQ(\alpha(t),p(t)) + \V_\HH(p;0,t) - \langle \LL(t), u(t) -w(t) \rangle \\
& \leq \QQ(e_0) +D(\alpha_0)+\| \nabla \alpha_0\|_2^2 + \tQ(\alpha_0,p_0) - \langle \LL_0, u_0 -w_0 \rangle +\int_0^{t} \langle \sigma(s), \EE \dot{w}(s) \rangle \,\mathrm{d}s \\& \hspace{1em} - \int_0^{t} \langle \dot{\LL}(s), u(s) - w(s) \rangle \,\mathrm{d}s.
\end{split}
\end{equation}
All the terms in the left-hand side are lower semicontinuous with respect to the convergences deriving from the boundedness of the functional and from the hypotheses on the loading (recall Theorem~\ref{thm:1505212332} for $\tQ$ and the lower semicontinuity of the plastic dissipation, which is supremum of lower semicontinuous functionals). As for the right-hand side, the integrals pass to the limit by Dominated Convergence Theorem in the time interval $[0,t]$ (notice that in the discrete inequalities the time interval is $[0, t_k^{i}]$, being $i(t,k)$ the integer part of $\frac{t}{T}k$, with $t_k^{i}\to t$ as $k\to +\infty$).

Let us prove the opposite energy inequality. Fix $t \in \ot$ and let $(s_k^i)_{0\le i\le k}$ be a sequence of subdivisions of the interval $[0,t]$ satisfying
\begin{eqnarray*}
& 0=s_k^0<s_k^1<\dots<s_k^{k-1}<s_k^k=t\,,\\
&\displaystyle
\lim_{k\to\infty}\,
\max_{1\le i\le k} (s_k^i-s_k^{i-1})= 0\,.
\end{eqnarray*}
{}For every $i=1,\ldots,k$ let $u:=u(s_k^i)-w(s_k^i)+ 
w(s_k^{i-1})$ and $e:=e(s_k^i)-Ew(s_k^i)+ 
Ew(s_k^{i-1})$. Since $\alpha(s_k^i)\leq\alpha(s_k^{i-1})$ and $(u, e, p(s_k^i))\in \A(w(s_k^{i-1}))$, by the
global stability \eqref{2205210000} we have
\begin{equation}\label{2305211324}
\begin{split}
\QQ(e(s_k^{i-1}))&+D(\alpha(s_k^{i-1}))+ \| \nabla \alpha(s_k^{i-1})\|_2^2 + \tQ(\alpha(s_k^{i-1}),p(s_k^{i-1})) + \HH(p(s_k^i)-p(s_k^{i-1})) \\& \hspace{1em}- \langle \LL(s_k^{i-1}), u(s_k^{i-1}) \rangle 
\\&\leq  \QQ(e(s_k^i)) - \langle \sigma(s_k^i), \EE(w(s_k^i)-w(s_k^{i-1}))\rangle +\gamma_2\| \EE(w(s_k^i)-w(s_k^{i-1}))\|_2^2 + D(\alpha(s_k^i))\\& 
\hspace{1em} + \| \nabla \alpha(s_k^i)\|_2^2 
 + \tQ(\alpha(s_k^i),p(s_k^i)) - \langle \LL(s_k^{i-1}), u(s_k^i) - (w(s_k^i)-w(s_k^{i-1})) \rangle.
\end{split}
\end{equation}
Rewriting the last term in the right hand side as
\begin{equation}\label{2305211430}
\begin{split}
\langle \LL(s_k^{i-1}), u(s_k^i) & - (w(s_k^i)-w(s_k^{i-1})) \rangle= \langle \LL(s_k^i), u(s_k^i) \rangle - \langle \LL(s_k^i), w(s_k^i) \rangle \\&
\hspace{1em} + \langle \LL(s_k^{i-1}), w(s_k^{i-1}) \rangle - \int_{s_k^{i-1}}^{s_k^i} \langle \dot{\LL}(s), u(s_k^i) - w(s_k^i) \rangle \,\mathrm{d}s,
\end{split}
\end{equation}
and summing \eqref{2305211324} over $i$, we obtain 
\begin{equation}\label{2305211454}
\begin{split}
\QQ(e_0)&+D(\alpha_0)+ \| \nabla \alpha_0\|_2^2 + \tQ(\alpha_0,p_0) + \V_\HH(p;0,t)- \langle \LL_0, u_0-w_0 \rangle 
\\&\leq  \QQ(e(t)) - \int_0^t \langle \ol \sigma_k(s), \EE(\dot{w}(s))\rangle\,\mathrm{d}s  + \delta'_k + D(\alpha(t))\\& 
\hspace{1em} + \| \nabla \alpha(t)\|_2^2 
 + \tQ(\alpha(t),p(t)) - \int_0^t \langle \dot{\LL}(s), \ol u_k(s) - \ol w_k(s) \rangle \, \mathrm{d}s.
\end{split}
\end{equation}
where $\delta_k'$ is defined similarly to $\delta_k$ (cf.\ \eqref{2305211506}) and we set $\ol u_k(s):= u(s_k^i)$, $\ol \sigma_k(s):=\sigma(s_k^i)$, $\ol w_k(s):=w(s_k^i)$ for $i$ the smallest index such that $s\leq s_k^i$.
By \eqref{2305211732}, \eqref{2305211828}, and the uniform bounds \eqref{1905211744} we get that 
\[
\langle \ol \sigma_k(s), \EE(\dot{w}(s))\rangle \to \langle \sigma(s), \EE(\dot{w}(s))\rangle,\]
\[
\langle \dot{\LL}(s), \ol u_k(s) - \ol w_k(s) \rangle \to \langle \dot{\LL}(s),  u(s) -  w(s) \rangle
\]
as $k\to +\infty$ for a.e.\ $s\in \ot$. Again by \eqref{1905211744}, we can apply Dominated Convergence Theorem to pass to the limit in \eqref{2305211454}.
Thus we deduce the energy inequality opposite to \eqref{2205210006}, and so
\begin{equation*}
\begin{split}
\QQ(e(t))& +D(\alpha(t))+\| \nabla \alpha(t)\|_2^2 + \tQ(\alpha(t),p(t)) + \V_\HH(p;0,t) - \langle \LL(t), u(t) -w(t) \rangle \\
& = \QQ(e_0) +D(\alpha_0)+\| \nabla \alpha_0\|_2^2 + \tQ(\alpha_0,p_0) - \langle \LL_0, u_0 -w_0 \rangle +\int_0^{t} \langle \sigma(s), \EE \dot{w}(s) \rangle \,\mathrm{d}s \\& \hspace{1em} - \int_0^{t} \langle \dot{\LL}(s), u(s) - w(s) \rangle \,\mathrm{d}s.
\end{split}
\end{equation*}
In view of the integration by parts
\begin{equation*}
\int_0^t \{ \langle \dot{\LL}(s), w(s) \rangle + \langle \LL(s), \dot{w}(s) \rangle  \} \,\mathrm{d}s = \langle \LL(t), w(t) \rangle - \langle \LL_0, w_0 \rangle
\end{equation*}
the energy balance above is equivalent to  
\begin{equation}\label{2305211841}\tag{EB}
\begin{split}
\QQ(e(t))& +D(\alpha(t))+\| \nabla \alpha(t)\|_2^2 + \tQ(\alpha(t),p(t)) + \V_\HH(p;0,t) - \langle \LL(t), u(t) \rangle \\
& = \QQ(e_0) +D(\alpha_0)+\| \nabla \alpha_0\|_2^2 + \tQ(\alpha_0,p_0) - \langle \LL_0, u_0 \rangle +\int_0^{t} \langle \sigma(s), \EE \dot{w}(s) \rangle \,\mathrm{d}s \\& \hspace{1em} - \int_0^{t} \{ \langle \dot{\LL}(s), u(s) \rangle + \langle \LL(s), \dot{w}(s) \rangle \} \,\mathrm{d}s.
\end{split}
\end{equation}
\subsection{Strong continuity a.e.\ in time for $\alpha$ and $e$} 
Evaluating the energy balance at two times $s$ and $t$, with $s<t$, we get (it is immediate to see that $\V_{\HH}(p; 0, t)=\V_{\HH}(p; 0, s)+\V_{\HH}(p; s, t)$)
\begin{equation*}
\begin{split}
\QQ(e(t))& +D(\alpha(t))+\| \nabla \alpha(t)\|_2^2 + \tQ(\alpha(t),p(t)) + \V_\HH(p;s,t) - \langle \LL(t), u(t) \rangle \\
& = \QQ(e(s)) +D(\alpha(s))+\| \nabla \alpha(s)\|_2^2 + \tQ(\alpha(s),p(s)) - \langle \LL(s), u(s) \rangle +\int_s^t \langle \sigma(\tau), \EE \dot{w}(\tau) \rangle \,\mathrm{d}\tau \\& \hspace{1em} - \int_s^t \{ \langle \dot{\LL}(\tau), u(\tau) \rangle + \langle \LL(\tau), \dot{w}(\tau) \rangle \} \,\mathrm{d}\tau.
\end{split}
\end{equation*} 
Let us fix $t\in [0,T]$ satisfying three sets of conditions: first, to be a time of weak or weak$^*$ continuity of the variables $\alpha$, $u$, $e$, $p$; second, to be a continuity point for the increasing function $s\mapsto \V_{\HH}(p; 0,s)$; third, to be a Lebesgue point for the time derivatives of the external loadings $\LL$, $w$.
Notice that any $t \in [0,T]$ except countable many satisfies the above conditions.

Then, as $s\to t$, we obtain that
\begin{equation}\label{0708211053}
\begin{split}
\QQ(e(t))& +D(\alpha(t))+\| \nabla \alpha(t)\|_2^2 + \tQ(\alpha(t),p(t))  \\
& = \lim_{s\to t} \Big\{\QQ(e(s)) +D(\alpha(s))+\| \nabla \alpha(s)\|_2^2 + \tQ(\alpha(s),p(s))\Big\}.
\end{split}
\end{equation} 
Due to the weak or weak$^*$ continuity of $\alpha$, $u$, $e$, $p$ in $t$, and by Theorem~\ref{thm:1505212332}, each of the four terms $\QQ(e(\cdot))$, $D(\alpha(\cdot))$, $\| \nabla \alpha(\cdot)\|_2^2$, $\tQ(\alpha(\cdot),p(\cdot))$ is lower semicontinuous as $s\to t$. By \eqref{0708211053} we deduce that these four terms are actually continuous as $s\to t$. Using the convexity of $\QQ$ and the fact that $\alpha(s) \to \alpha(t)$ in $L^1(\Omega)$, we get that $e(s) \to e(t)$ in $\Lnn$ and $\alpha(s) \to \alpha(t)$ in $H^1(\Omega)$.
Further, 
\begin{equation*}
\tQ(\alpha(s),p(s)) \to \tQ(\alpha(t),p(t)),
\end{equation*}
which can be read as $\sqrt{C_1(\alpha(s))}p(s) \to \sqrt{C_1(\alpha(t))}p(t)$ in $L^2(\Omega; \Mnn)$.

This concludes the proof of Theorem~\ref{thm:2305211857}.

\subsection{From weak to strong evolutions under further regularity}

In this subsection we link the notion of evolution in Theorem~\ref{thm:2305211857}, which has an integral formulation, with that one in \cite{KazMar18} (with the introduction of damage gradient, see Remark~3.1 therein), corresponding to a differential formulation for any time.
We refer to the latter as \emph{strong formulation}, whose solutions are \emph{strong evolutions}. In fact, it could be seen that strong evolutions satisfy also the conditions in Theorem~\ref{thm:2305211857}.
The converse is true only under additional regularity assumptions. We first describe the differential properties that could be deduced without further assumptions. 

\begin{prop}\label{prop:0708211246}
For every evolution satisfying the conditions of Theorem~\ref{thm:2305211857}, under the corresponding hypotheses, for every $t \in [0,T]$ it holds that 
\begin{equation*}
-\diver \sigma(t)=f(t),\qquad [\sigma(t) \nu]=g(t),\qquad \sigma(t)-2 \C_1(\alpha(t))p(t) \in K \ \text{a.e.\ in }\Omega,
\end{equation*}
for $\sigma(t):=\C e(t)$, where $ [\sigma(t) \nu] \in H^{-\frac12}(\Omega)$ is defined as done for $\varrho$ after \eqref{0905211302}.
\end{prop}
\begin{proof}
For every $t\in [0,T]$, fixing $\alpha(t)=\beta$ in the minimum problem (qs1) we obtain that $(u(t), e(t), p(t))\in \A(w(t))$ solves the minimisation problem
 \begin{equation}\label{0708211331}
\begin{split}
\min_{ (v,\eta,q)\in \A(w(t)) } \Big\{ \QQ(\eta) + \tQ(\alpha(t), q)  +\HH(q-p(t)) - \langle \LL(t), v \rangle \Big\}.
\end{split}
\end{equation}
Therefore, we are in a situation analogous to that of \cite[Lemma~3.6]{Mora}, except for the presence of $\tQ$ and the fact that $K\subset \Mnn$ (instead of $K \subset \MD$).
The derivation of the first two conditions $-\diver \sigma(t)=f(t)$, $[\sigma(t) \nu]=g(t)$ goes exactly as in \cite[Lemma~3.6]{Mora}, since the test directions are $(v, \EE v, 0)$, with $v \in H^1(\Omega;\Rn)$, so that the contribution of $\tQ$ disappears.

In order to show the last condition, we test \eqref{0708211331} with $(u(t), e(t), p(t)) + \varepsilon (0, \eta, -\eta)$, with $\eta \in \Lnn$, and derive with respect to $\varepsilon$. 
Here we have the further term $\tQ$ with respect to \cite[Lemma~3.6]{Mora}: it holds that
\[
\lim_{\varepsilon\to 0}\frac{\tQ(\alpha(t), p(t) -\varepsilon \eta)- \tQ(\alpha(t), p(t))}{\varepsilon}
= -2 \int_{\Omega} \C_1(\alpha(t)) p(t) \, \eta \,\dx, \]
by Dominated Convergence Theorem, since both $\C_1(\alpha(t)) p(t)$ and $\eta \in \Lnn$.
At this stage, it is enough to argue as in the last part of the proof of \cite[Proposition~3.5]{Mora}, by choosing $\eta(x)= 1_B(x) \xi$, $\xi \in \Mnn$, since $\partial H(0)=K$. 
\end{proof}

The remaining properties, consisting in flow rules for the damage and the plasticity variable, may require some strong further regularity in time, to guarantee global differentiability in time: even if this regularity could be hard to prove, we list it to confirm that the evolutions whose existence has been proven here
are the right weak evolutions for the model in \cite{KazMar18}. We start with a technical lemma.

\begin{lemma}\label{le:0808211330}
Besides the assumptions in Theorem~\ref{thm:2305211857}, let us assume that $d \in C^1([0,1])$, $\C_1 \in C^1([0,1); Lin(\Mnn\times\Mnn))$ with $|\C'_1(\beta)|\leq C\,\C_1(\beta)$ for some constant $C>0$ and every $\beta\in[0,1]$. Let $\alpha$, $u$, $e$, $p$ be an evolution according to Theorem~\ref{thm:2305211857}. Then, for every $\beta \in L^\infty(\Omega)$ such that $\beta\leq 0$ and $\beta =0$ if $\alpha=0$, and for every $t \in [0,T]$, it holds that
\begin{equation*}
\begin{split}
\partial_\alpha \En(\alpha(t), e(t), p(t))[\beta]= 
\langle d'(\alpha(t)), \beta \rangle + \langle \nabla \alpha(t), \nabla \beta \rangle +  \int_{\Omega\sm\{\alpha(t)=1\}} \hspace{-1em} \C_1'(\alpha(t)) \beta \,p(t)\colon p(t) \dx \geq 0,
\end{split}
\end{equation*}
for $\En(\alpha(t), e(t), p(t)):= \QQ(e(t))+ D(\alpha(t)) + \| \nabla \alpha(t)\|_2^2 +\tQ(\alpha(t), p(t))$ and
\begin{equation*}
\partial_\alpha \En(\alpha(t), e(t), p(t))[\beta]:= \lim_{\varepsilon\to 0^+} \frac{\En(\alpha(t)+\varepsilon \beta, e(t), p(t))-\En(\alpha(t), e(t), p(t))}{\varepsilon}.
\end{equation*}
\end{lemma}
\begin{proof}
Since $\beta \leq 0$, the stability condition (qs1) gives that (notice that all the terms depending on $\alpha(t)$ are included in $\En$), for $\varepsilon>0$,
\begin{equation}\label{1208212021}
\frac{\En(\alpha(t)+\varepsilon \beta, e(t), p(t))-\En(\alpha(t), e(t), p(t))}{\varepsilon} \geq 0.
\end{equation}
In order to pass to the limit above, notice that $\frac{\mathrm{d}}{\mathrm{d}\varepsilon}\big |_{\varepsilon=0^+} \C_1((\alpha(t) + \varepsilon\beta)^+) (x)= \C_1'(\alpha(t)) \beta(x)$ for a.e.\ $x \in \Omega$ such that $\alpha(t)(x)<1$, since this holds if $\alpha(t)(x)=0$, being $\beta=0$ if $\alpha=0$, and also if $\alpha(t)(x)>0$, since $\beta \in L^\infty(\Omega)$ and $\varepsilon$ tends to 0. If $\alpha(t)(x)=1$, then $p(t)(x)=0$ and $[\C_1((\alpha(t) + \varepsilon\beta)^+) p(t) \colon p(t)] (x)=[\C_1(\alpha(t)) p(t) \colon p(t)] (x)=0$. Therefore \eqref{1208212021} holds true, and the integrals defining $\tQ$ are restricted to $\{\alpha(t)<1\}$.

For a.e.\ $x\in \{\alpha(t)<1\}$ it holds that $\varepsilon^{-1}[\C_1(\alpha(t) +\varepsilon \beta)^+)(x)-\C_1(\alpha(t))(x)]=\C'_1(\alpha(t)+\varepsilon' \beta)\leq \C_1(\alpha(t))(x)$ for some $\varepsilon'\in (0,\varepsilon)$ and $\varepsilon$ small enough, using Intermediate Value Theorem, the fact that $\beta \in L^\infty(\Omega)$, \eqref{C3}, and the further assumptions on $\C_1$.
Therefore, we can pass to the limit by Dominated Convergence Theorem, to differentiate the term in $\tQ$.
  The derivation of the remaining two terms is straightforward.
\end{proof}
\begin{rem}\label{rem:1308210846}
The assumptions on $\C_1$ in Lemma~\ref{le:0808211330} are satisfied, for instance, if either $\C_1 \in C^1([0,1]; Lin(\Mnn\times\Mnn))$ or $\C_1$ diverges in 1 with an exponential growth. Unfortunately, the power law considered in \cite{KazMar18} does not satisfy the assumptions. Anyway, other conditions allowing to compute directional derivatives w.r.t.\ the damage variable could be found also in this case, for instance concerning the regularity of $\alpha(t)$, $p(t)$. We remark again that this last part of the paper aims to convince that, under further regularity or reasonable constitutive assumptions, we recover a strong evolution.
\end{rem}

In the next proposition we derive the flow rules, provided the evolution is regular enough.
\begin{prop}\label{prop:0808211721}
Let $\alpha$, $u$, $e$, $p$ be an evolution according to Theorem~\ref{thm:2305211857}. Besides the assumptions in Theorem~\ref{thm:2305211857} and in Lemma~\ref{le:0808211330}, 
let us assume 
that $\alpha$, $u$, $e$, $p$ are absolutely continuous with values into $H^1 \cap L^\infty$ (we mean with respect to both norms), $BD(\Omega)$, $\Lnn$, $\MbDd$, respectively. Moreover, assume for a.e.\ $t\in [0,T]$ that $\sigma(t)$ is integrable with respect to $\dot{p}(t)$ and that it holds the integration by parts formula
\begin{equation}\label{0808211726}
\langle \sigma(t), \dot{e}(t)-\EE \dot{w}(t) \rangle + \langle \LL(t), \dot{w}(t)-\dot{u}(t) \rangle + \int_{\Omega\cup \dod} \sigma(t) \colon \mathrm{d}\dot{p}(t)=0.
\end{equation}
Then for a.e.\ $t\in [0,T]$ the following flow rules hold:
\begin{equation}\label{0808211759}
\HH(\dot{p}(t))=\int_{\Omega\cup \dod} \sigma(t) \colon \mathrm{d}\dot{p}(t)\quad\text{and}\quad \partial_\alpha \En(\alpha(t), e(t), p(t))[\dot{\alpha}(t)]=0.
\end{equation}
\end{prop}
\begin{proof}
By the assumptions we made on the evolution, we have that $(\dot{u}(t), \dot{e}(t), \dot{p}(t)) \in \A(\dot{w}(t))$ for a.e.\ $t$ (cf.\ \cite[Lemma~5.5]{Mora}) and that
we may differentiate in time the energy balance condition (qs2), to get
\begin{equation}\label{0808211756}
\langle \sigma(t), \dot{e}(t)-\EE \dot{w}(t) \rangle + \langle \LL(t), \dot{w}(t)-\dot{u}(t) \rangle + \HH(\dot{p}(t)) + \partial_\alpha \En(\alpha(t), e(t), p(t))[\dot{\alpha}(t)]=0.
\end{equation}
In fact, notice that $\dot{\alpha}(t)=0$ if $\alpha(t)=0$ where $\dot{\alpha}(t)=0$ exists, since $\alpha(s)=0$ for $s\geq t$. In view of (the last equality in) \eqref{2607210848} and by Lemma~\ref{le:0808211330}, the left handside of \eqref{0808211756} is the sum of two nonnegative terms. Therefore \eqref{0808211756} implies \eqref{0808211759}. 
\end{proof}

\begin{rem}
The condition \eqref{0808211726} is the analogue of \eqref{1605211221}. It holds under additional regularity either for $\sigma(t)$, e.g.\ the same regularity required on $\varrho(t)$, recalling that $(\dot{u}(t), \dot{e}(t), \dot{p}(t)) \in \A(\dot{w}(t))$, or for $\dot{u}(t)$. For instance, in the Dirichlet loading case it is enough to have either $\sigma(t) \in L^n(\Omega;\Mnn)$ or $\dot{u}(t) \in L^2(\Omega;\Rn)$, see \cite[Theorem~6.4]{BabMor15}. 
 We do not go further in detail about these conditions since we are already in the strong assumption that the evolution is absolutely continuous in time. Moreover, notice that the plastic flow rule implies that $\mathbf{H}(\dot{p}(t))=\sigma\colon \dot{p}(t)$ in $\MbDd$, by \eqref{2607210847}.
\end{rem}

\noindent {\bf Acknowledgments.} {The author wishes to thank Stefano Vidoli for 
 fruitful 
discussions.}

\end{document}